\documentclass[review,onefignum,onetabnum]{siamart171218}

\usepackage{lipsum}
\usepackage{mathrsfs}
\usepackage{amsfonts}
\usepackage{graphicx}
\usepackage{epstopdf}
\usepackage{algorithmic}
\usepackage{amssymb}
\usepackage{amsmath}

\ifpdf
  \DeclareGraphicsExtensions{.eps,.pdf,.png,.jpg}
\else
  \DeclareGraphicsExtensions{.eps}
\fi

\newsiamremark{hypothesis}{Hypothesis}
\crefname{hypothesis}{Hypothesis}{Hypotheses}
\newsiamthm{claim}{Claim}
\newtheorem{remark}{Remark}
\nolinenumbers

\headers{Lommel, Anger-Weber and Struve functions}{T. M. Dunster}

\title{Uniform asymptotic expansions for Lommel, Anger-Weber and Struve functions}

\author{T. M. Dunster\thanks{Department of Mathematics and Statistics, San Diego State University, 5500 Campanile Drive, San Diego, CA 92182-7720, USA. 
  (\email{mdunster@sdsu.edu}, \url{https://tmdunster.sdsu.edu}).}}

\usepackage{amsopn}

\makeatletter
\newcommand*{\addFileDependency}[1]{
  \typeout{(#1)}
  \@addtofilelist{#1}
  \IfFileExists{#1}{}{\typeout{No file #1.}}
}
\makeatother

\nolinenumbers

\ifpdf
\hypersetup{
  pdftitle={Uniform asymptotic expansions for Lommel, Anger-Weber and Struve functions},
  pdfauthor={T. M. Dunster}
}

\begin{document}

\maketitle

\begin{abstract}
  Using a differential equation approach asymptotic expansions are rigorously obtained for Lommel, Weber, Anger-Weber and Struve functions, as well as Neumann polynomials, each of which is a solution of an inhomogeneous Bessel equation. The approximations involve Airy and Scorer functions, and are uniformly valid for large real order $\nu$ and unbounded complex argument $z$. An interesting complication is the identification of the Lommel functions with the new asymptotic solutions, and in order to do so it is necessary to consider certain sectors of the complex plane, as well as introduce new forms of Lommel and Struve functions.
\end{abstract}



\section{Introduction} 
\label{sec1}

One of the most important classes of special functions are Bessel functions, which have numerous mathematical and physical applications. In this paper we undertake a comprehensive study of the asymptotic behaviour of solutions of the inhomogeneous form of Bessel's equation
\begin{equation}  \label{02}
\frac{d^{2}y}{dz^{2}}+\frac{1}{z}\frac{dy}{dz}
+\left(1-\frac{\nu^{2}}{z^{2}} \right)y=z^{\mu-1},
\end{equation}
for large real $\nu$, bounded complex $\mu$, and unbounded complex $z$.

The homogeneous form of (\ref{02}) has extensively-studied standard solutions $J_{\nu}(z)$, $H_{\nu}^{(1)}(z)$, $ H_{\nu}^{(2)}(z)$, which for $\nu>0$ form a numerically satisfactory set of solutions in the principal plane $-\pi<\arg(z)\leq \pi$ since they are recessive at $z=0$, $z=\infty$ ($\Im(z)>0$), and $z=\infty$ ($\Im(z)<0$), respectively. In particular
\begin{equation}  \label{46}
J_{\nu}(z) \sim 
\frac{\left(\frac{1}{2}z\right)^{\nu}}{\Gamma(\nu+1)}
\quad (z \rightarrow 0),
\end{equation}
\begin{equation}  \label{46a}
H^{(1)}_{\nu}(z) \sim
\left(\frac{2}{\pi z}\right)^{1/2}
e^{i\left(z-\frac{1}{2}\nu\pi-\frac{1}{4}\pi\right)}
\quad (z \rightarrow \infty, \, -\pi+\delta \leq
\arg(z) \leq 2\pi-\delta),
\end{equation}
and
\begin{equation}  \label{46b}
H^{(2)}_{\nu}(z) \sim
\left(\frac{2}{\pi z}\right)^{1/2}
e^{-i\left(z-\frac{1}{2}\nu\pi-\frac{1}{4}\pi\right)}
\quad (z \rightarrow \infty, \, -2\pi+\delta \leq
\arg(z) \leq \pi-\delta),
\end{equation}
where throughout this paper $\delta \in (0,1)$ represents an arbitrary small constant, unless otherwise stated. Uniform asymptotic expansions for large $\nu$ and complex $z$ for these functions are well-established. See \cite[Chap. 11]{Olver:1997:ASF} for expansions valid for unbounded complex $z$ involving Airy functions, complete with explicit error bounds. More recently in \cite{Dunster:2017:COA} and \cite{Dunster:2020:SEB} uniform asymptotic expansions, also involving Airy functions and valid in similar domains, were derived that are easier to compute, with simpler error bounds. These results will play an important part of the new expansions in this paper.

Standard particular solutions of (\ref{02}) are the Lommel functions denoted by $s_{\mu,\nu}(z)$ and $S_{\mu,\nu}(z)$. Their asymptotic approximation is much less well-developed, and the purpose here is to fill this gap. We shall obtain uniform asymptotic expansions for them, as well as newly defined numerically satisfactory solutions, for $\nu$ real and large, and bounded $\mu \in \mathbb{C}$. The new expansions are valid for unbounded real or complex $z$, and in conjunction with appropriate connection formulas valid for unrestricted $\arg(z)$. In short we bring the asymptotic theory of Lommel functions on par with their Bessel function counterparts.

In addition we apply our new results to a number of important related functions, namely the Anger function $\mathbf{J}_{\nu}(z)$, Weber function $\mathbf{E}_{\nu}(z)$, Anger-Weber function $\mathbf{A}_{\nu}(z)$, Struve functions $\mathbf{H}_{\nu}(z)$ and $\mathbf{K}_{\nu}(z)$, and Neumann polynomials $O_{n}(z)$. The domains of validity are as large as our ones for the Lommel functions.

Lommel functions and the related ones listed above occur in fluid dynamics, aerodynamics, magnetohydrodynamics, optical diffraction, and particle quantum dynamics. For a list of references see \cite[Sect. 11.2]{NIST:DLMF}. We also mention that Lommel functions can be expressed explicitly in terms of integrals, as well as hypergeometric functions \cite[Sect. 10.7]{Watson:1944:TTB}. They are also solutions of a homogeneous third order linear differential equation, which can be derived from differentiating (\ref{02}). The Neumann polynomials that we approximate are used to expand arbitrary functions as a series of Bessel functions \cite[Sect. 10.23(iii)]{NIST:DLMF}.

As another application, in \cite[Sect. 10.74]{Watson:1944:TTB} we find that some important integrals can be expressed in terms of Lommel functions, specifically
\begin{equation}
\int {z}^{\mu}\mathscr{C}_{\nu}(z) dz
= (\mu+\nu-1) z \mathscr{C}_{\nu}(z) S_{\mu-1,\nu-1}(z) 
-z \mathscr{C}_{\nu-1}(z) S_{\mu,\nu}(z),
\end{equation}
where $\mathscr{C}_{\nu}(z)$ is any solution of Bessel's equation.

For a summary of known asymptotic results of solutions of (\ref{02}) see \cite[Sects. 11.6 and 11.11]{NIST:DLMF} and the references therein. Most of these come from integral methods. For recent work we mention \cite{Nemes:2014:RP1} and \cite{Nemes:2014:RP2} where Nemes considered asymptotic expansions due to Watson \cite[Sect. 10.15]{Watson:1944:TTB} for the Anger-Weber functions $\mathbf{A}_{\pm \nu}(z)$ for large complex $\nu$. Using a reformulation of the method of steepest descents he constructed explicit and realistic error bounds, and asymptotics for the late coefficients are obtained, along with exponentially improved asymptotic expansions. In \cite{Nemes:2020:ELM} an approximation for $\mathbf{A}_{- \nu}(z)$ given by \cite[Eq. 11.11.17]{NIST:DLMF} is extended to a full expansion, but this is only valid in a neighbourhood of a turning point. 

A major difference is that in those papers $\arg(\nu)=\arg(z)$, whereas our corresponding results (in \cref{sec4}) have $\nu$ real and $z$ complex. Also our expansions are uniformly valid for complex $z$ lying in an unbounded domain arbitrarily close to $z=0$, and are also valid at $z=\infty$ ($|\arg(z)|\leq \pi -\delta$). For $\mathbf{A}_{-\nu}(z)$ our domain of validity is larger, including a full neighbourhood of a turning point ($z=\nu$), which in existing approximations must be treated separately.

In \cite[Chap. 9, Sect. 12]{Olver:1997:ASF} Olver obtains asymptotic expansions for $\mathbf{A}_{-\nu}(x)$ for large real $\nu$ using integral methods, which are valid for $0<x\leq 1$ and $1 \leq x <\infty$. These expansions involve a sequence of integrals, and are more complicated than our results, which as we mentioned are also not restricted to real argument.

In \cite{Paris:2015:ASF} Paris examined Watson's expansions of $\mathbf{H}_{\nu}(z)$ \cite[Sect. 10.43]{Watson:1944:TTB} for certain complex values of $\nu/z$. In \cite{Paris:2018:AGS} he obtained asymptotic expansions of a generalised Struve function for large complex argument. In \cite{Nemes:2015:OLA} and \cite{Nemes:2018:EBF} Nemes studied the the Lommel and related functions in detail for large $z$ and bounded parameters, deriving explicit error bounds, exponentially improved asymptotic expansions, and the smooth transition of the Stokes discontinuities.

The plan of this paper is as follows. In \cref{sec2} we give definitions and connection formulas for the Lommel functions. We also define three new such functions which are numerically satisfactory in various parts of the complex plane, which we denote by $S_{\mu,\nu}^{(j)}(z)$ ($j=0,1,2$). We provide identities that will be important for our subsequent asymptotic approximations, such as connection relations between these functions and Bessel functions, and analytic continuation formulas.

In \cref{sec3} uniform asymptotic approximations are derived for the Lommel functions $S_{\mu,\nu}(z)$ and $S_{\mu,\nu}^{(j)}(z)$ ($j=0,1,2$), with corresponding expansions for $s_{\mu,\nu}(z)$ following from these and a connection formula. The new expansions involve the same variables as for standard Bessel functions used in a Liouville transformation. The principal approximations come from the new general theory given in \cite{Dunster:2020:ASI} for inhomogeneous differential equations having a turning point. These asymptotic solutions involve slowly varying coefficient functions that are easy to compute, and Scorer functions which are solutions of an inhomogeneous Airy equation. Our new expansions are valid for large real $\nu$, fixed $\mu \in \mathbb{C}$, and uniformly for complex $z \neq 0$ (which is otherwise unrestricted via appropriate connection formulas). The Scorer function expansions are valid in a domain containing $|\arg(z)| \leq \pi - \delta < \pi$, in particular a full neighbourhood of the turning point in question. We also provide simpler expansions which are valid in unbounded, but more restricted parts of the complex plane, that do not contain a turning point.

In \cref{sec4} we apply the results of \cref{sec3} to provide uniform asymptotic expansions for the functions $\mathbf{J}_{\pm \nu}(z)$, $\mathbf{E}_{\pm \nu}(z)$, $\mathbf{A}_{\pm \nu}(z)$ and $O_{n}(z)$. In \cref{sec5} we similarly obtain expansions for $\mathbf{H}_{\nu}(z)$ and $\mathbf{K}_{\nu}(z)$. The expansions of \cref{sec4,sec5}, in conjunction with analytic continuation and connection formulas, are valid for large real $\nu$ (and $n$ for the Neumann polynomials), uniformly for unrestricted complex $z \neq 0$. 

\section{Lommel functions: Definitions and connection formulas} 
\label{sec2}

The first Lommel function is defined by \cite[Eqs. 11.9.3 and 11.9.4]{NIST:DLMF}

\begin{equation}  \label{03}
s_{\mu,\nu}(z)=
z^{\mu+1}\sum_{k=0}^{\infty}(-1)^{k}
\frac{z^{2k}}{a_{k+1}(\mu,\nu)},
\end{equation}
where
\begin{equation}  \label{04}
a_{k}(\mu,\nu)=\prod_{m=1}^{k}\left\{(\mu+2m-1)^{2}-\nu^{2}\right\},
\end{equation}
and principal values correspond to principal values of $(z/2)^{\mu+1}$. Its significance is that it is a solution which is real-valued for $\mu,\,\nu \in \mathbb{R}$ and $z>0$, and is bounded at $z=0$ when $\Re(\mu) \geq -1$. However, it is not defined if $\pm \nu - \mu = 1,3,5,\cdots$, which is a severe limitation. Another problem is that the function is not uniquely defined by its behaviour at $z=0$ when $\Re(\nu) > \Re(\mu)+1$, since from (\ref{46}) we see that $s_{\mu,\nu}(z)+CJ_{\nu}(z)$ is also $\mathcal{O}(z^{\mu+1})$ as $z \rightarrow 0$ for arbitrary constant $C$. We shall shortly introduce new Lommel functions that are subdominant at $z=0$ and do not have these deficiencies.

From \cite[p. 346]{Watson:1944:TTB} we note the integral representation
\begin{equation}  \label{05}
 s_{\mu,\nu}(z)
 =\frac{\pi}{2}
 \left[Y_{\nu}(z)\int_{0}^z 
 t^{\mu}J_{\nu}(t)dt  -J_{\nu}(z)\int_{0}^z 
 t^{\mu}Y_{\nu}(t)dt  \right]   \quad (\Re(\mu \pm \nu)>-1).
\end{equation}

Next from \cite[Eq. 11.9.5]{NIST:DLMF} a companion solution is defined by
\begin{equation}  \label{06}
S_{\mu,\nu}(z)=s_{\mu,\nu}(z)+A(\mu,\nu)\left\{\sin\left(\tfrac{1}{2}(\mu-\nu)\pi \right )J_{\nu}(z)-\cos\left(\tfrac{1}{2}(\mu-\nu)\pi \right )Y_{\nu}(z)\right\},
\end{equation}
where
\begin{equation}  \label{07}
A(\mu,\nu)=2^{\mu-1}\Gamma\left(\tfrac{1}{2}\mu+\tfrac{1}{2}\nu+\tfrac{1}{2}\right)\Gamma\left(\tfrac{1}{2}\mu-\tfrac{1}{2}\nu+\tfrac{1}{2}\right).
\end{equation}

We note the following reflection relations \cite[Eq. 11.9.6]{NIST:DLMF} shared by both of these Lommel functions
\begin{equation}  \label{07a}
s_{\mu,-\nu}(z)=s_{\mu,\nu}(z), \quad
S_{\mu,-\nu}(z)=S_{\mu,\nu}(z),
\end{equation}
and we utilise these later.

From (\ref{07}) and the Gamma function reflection property \cite[Eq.5.5.3]{NIST:DLMF} we see that the coefficient of $Y_{\nu}(z)$ in (\ref{06}) is (ignoring the minus sign)
\begin{equation}  \label{07b}
A(\mu,\nu)\cos\left(\tfrac{1}{2}(\mu-\nu)\pi\right) 
=\frac{\pi 2^{\mu-1}\Gamma\left(\tfrac{1}{2}\mu+\tfrac{1}{2}\nu
+\tfrac{1}{2}\right)}{\Gamma\left(\tfrac{1}{2}\nu
-\tfrac{1}{2}\mu+\tfrac{1}{2}\right)}.
\end{equation}
Thus $S_{\mu,\nu}(z)$ is certainly unbounded at $z=0$ if this is non-zero, since so too is $Y_{\nu}(z)$ (see \cite[Eq. 10.7.1 and 10.7.4]{NIST:DLMF}). The case where this coefficient vanishes, namely $\nu-\mu=-1,-2,-5,\cdots$ is not applicable in this paper, since we are considering $\mu$ bounded and $\nu \rightarrow \infty$.

The characterising property of $S_{\mu,\nu}(z)$ is that as $z\rightarrow \infty$ with $|\arg(z)|\leq\pi-\delta$ (\cite[Eq. 11.9.9]{NIST:DLMF})
\begin{equation}  \label{08}
S_{\mu,\nu}(z) \sim z^{\mu-1}\sum_{k=0}^{\infty}(-1)^{k}a_{k}(-\mu,\nu)z^{-2k}.
\end{equation}
Hence it is the unique solution of (\ref{02}) that grows at worst algebraically at $z= \pm i \infty$ in the principal plane: all other solutions are exponentially large at least at one of these singularities. We call this function subdominant at these singularities. Uniqueness also follows from its non-oscillatory behaviour along the positive real axis ($\arg(z)=0$). Moreover, unlike $s_{\mu,\nu}(z)$, it is well-defined for all values of the parameters.

We remark that at a singularity ($z=0$ or $z=\infty$) homogeneous solutions of (\ref{02}) are either recessive or dominant; see \cite[Chap. 5, Sect. 7.2]{Olver:1997:ASF} for a general discussion. For $\Re(\nu) > \Re(\mu)+1$ particular solutions are not recessive at a singularity, but rather are either subdominant or dominant there. For example, $H_{\nu}^{(1)}(z)$, $J_{\nu}(z)$ and $S_{\mu,\nu}(z)$ is recessive, dominant and subdominant, respectively, at $z=\infty \exp(\pi i/2)$. Recessiveness at one singularity uniquely defines a homogeneous solution (up to a multiplicative constant), whereas (assuming $\Re(\nu) > \Re(\mu)+1$) subdominance at two singularities is required to uniquely define a particular solution. 

Returning to (\ref{09}), if either of $\mu \pm \nu$ equals an odd positive integer then the RHS terminates and it represents $S_{\mu,\nu}(z)$ exactly. As a special case we get the Neumann polynomials, which are considered in \cref{sec4}.

Using the Wronskian \cite[Eq. 10.5.5]{NIST:DLMF}
\begin{equation}  \label{09}
\mathscr{W}\left\{H_{\nu}^{(1)}(z),H_{\nu}^{(2)}(z)\right\}
=-4i/(\pi z),
\end{equation}
we have by variation of parameters on (\ref{02}), subdominance at infinity for $|\arg(z)|\leq \pi -\delta$, and referring to (\ref{46a}) and (\ref{46b})
\begin{multline}  \label{10}
S_{\mu,\nu}(z)
=\frac{\pi i}{4}
\left[H_{\nu}^{(2)}(z)\int_{\infty \exp(\pi i/2)}^z 
t^{\mu}H_{\nu}^{(1)}(t)dt  \right. \\
\left. -H_{\nu}^{(1)}(z)\int_{\infty \exp(-\pi i/2)}^z t^{\mu}H_{\nu}^{(2)}(t)dt
\right].
\end{multline}

A numerically satisfactory form of a general solution of (\ref{02})  in the cut complex plane ($|\arg(z)|\leq\pi$), which excludes a certain region containing $z=0$, is given by
\begin{equation}  \label{11}
y(\nu,z)=c_{1}H_{\nu}^{(1)}(z)+c_{2}H_{\nu}^{(2)}(z)+S_{\mu,\nu}(z).
\end{equation}

As we remarked, $s_{\mu,\nu}(z)$ is not a satisfactory particular solution when considering small $z$. Instead we introduce a new functions $S_{\mu,\nu}^{(j)}(z)$ ($j=0,1,2$) such that a general solution has the numerically satisfactory representation
\begin{equation}  \label{11a}
y(\nu,z)=c_{0}J_{\nu}(z)+c_{1}H_{\nu}^{(1)}(z)+S_{\mu,\nu}^{(1)}(z),
\end{equation}
for $0 < \arg(z) < \pi$, 
\begin{equation}  \label{11b}
y(\nu,z)=c_{0}J_{\nu}(z)+c_{2}H_{\nu}^{(2)}(z)+S_{\mu,\nu}^{(2)}(z),
\end{equation}
for $-\pi < \arg(z) < 0$,
\begin{equation}  \label{11c}
y(\nu,x)=c_{0}J_{\nu}(x)+c_{1}Y_{\nu}(x)+S_{\mu,\nu}^{(0)}(x),
\end{equation}
for $0<x \leq \nu$, and
\begin{equation}  \label{11d}
y(\nu,x)=c_{0}J_{\nu}(x)+c_{1}Y_{\nu}(x)+S_{\mu,\nu}(x),
\end{equation}
for $\nu \leq x< \infty$.
	
These new Lommel functions are defined as follows. Firstly if we add the homogeneous solution
\begin{multline}  \label{12}
-iA(\mu,\nu)\cos\left(\tfrac{1}{2}(\mu-\nu)\pi\right)
\left\{J_{\nu}(z)+iY_{\nu}(z)\right \} \\
=-iA(\mu,\nu)\cos\left(\tfrac{1}{2}(\mu-\nu)\pi\right)
H_{\nu}^{(1)}(z),
\end{multline}
to both sides of (\ref{12}). This then leads to our first definition, namely
\begin{equation}  \label{13}
S_{\mu,\nu}^{(1)}(z)
=S_{\mu,\nu}(z)-iA(\mu,\nu)
\cos\left(\tfrac{1}{2}(\mu-\nu)\pi\right)H_{\nu}^{(1)}(z),
\end{equation}
or equivalently, from (\ref{07b})
\begin{equation}  \label{14}
S_{\mu,\nu}^{(1)}(z)
=S_{\mu,\nu}(z)-\frac{i\pi 2^{\mu-1}\Gamma\left(\tfrac{1}{2}\mu+\tfrac{1}{2}\nu
+\tfrac{1}{2}\right)}{\Gamma\left(\tfrac{1}{2}\nu
-\tfrac{1}{2}\mu+\tfrac{1}{2}\right)}H_{\nu}^{(1)}(z).
\end{equation}
Moreover, from (\ref{06}), (\ref{12}) and (\ref{13}), we have
\begin{equation}  \label{15}
S_{\mu,\nu}^{(1)}(z)=s_{\mu,\nu}(z)-ie^{(\mu-\nu)\pi i/2}A(\mu,\nu)J_{\nu}(z).
\end{equation}

Thus this is a particular solution of (\ref{02}) that is subdominant at $z=\infty \exp(\pi i/2)$ by virtue of (\ref{46a}), and it has the same expansion (\ref{08}) at infinity as $S_{\mu,\nu}(z)$ in the upper half plane $0<\arg(z)|<\pi$. Significantly, from (\ref{15}) it is bounded at $z=0$ too if $\Re(\mu) \geq -1$ and $\nu \geq 0$. Moreover, although unbounded at $z=0$ if $\Re(\mu)<-1$ it is still subdominant at this singularity relative to all other solutions except $J_{\nu}(z)$ when $\nu$ is sufficiently large.

Thus $S_{\mu,\nu}^{(1)}(z)$ is subdominant at the two singularities $z=0$ and $z=\infty \exp(\pi i/2)$, and this uniquely defines it for any fixed $\mu \in \mathbb{C}$ and $\nu \rightarrow \infty$. We remark that, unlike $S_{\mu,\nu}(z)$, it is not subdominant at $z=\infty \exp(-\pi i/2)$.

Another important observation is that $S_{\mu,\nu}^{(1)}(z)$ is defined for a larger range of parameters than $s_{\mu,\nu}(z)$. Specifically, it is defined for all parameter values except when both $\nu + \mu = -1,-3,-5,\cdots$ and $\nu - \mu \neq -1,-3,-5,\cdots$. Importantly it is well-defined for the parameter range we consider in this paper, namely $\mu$ fixed and $\nu$ large and positive.

If $\nu>\mu+1$ the leading terms in its series expansion at $z=0$ is the same as those for $s_{\mu,\nu}(z)$, given by (\ref{03}) and (\ref{04}), with $k$ running from $0$ to $K$, where $\mu+1+2K<\nu$. If $\nu -\mu \neq 1,3,5,\cdots$ subsequent terms come from (\ref{03}) and the series for $J_{\nu}(z)$ (\cite[Eq. 10.2.2]{NIST:DLMF}) inserted in (\ref{15}). 

If $\nu - \mu$ is an odd integer limits can be taken. Thus for $\nu=\mu+2m+1$ ($m \in \mathbb{N}$) we have
\begin{multline}  \label{16}
S_{\mu,\mu+2m+1}^{(1)}(z)=\lim_{\epsilon \rightarrow 0}
\Big\{ s_{\mu,\mu+2m+1+\epsilon}(z) \\
\left.  +(-1)^{m+1}e^{-\pi i \epsilon/2}
A(\mu,\mu+2m+1+\epsilon)J_{\mu+2m+1+\epsilon}(z)
\right\}.
\end{multline}
See also \cite[Sect. 10.73]{Watson:1944:TTB} for a similar approach for $S_{\mu,\nu}(z)$. For example, if $\Re(\nu)>\Re(\mu)+3$ we have from (\ref{03}), (\ref{04}) and (\ref{15}) as $z \rightarrow 0$
\begin{equation}  \label{17}
 S_{\mu,\nu}^{(1)}(z) = \frac{z^{\mu+1}}{(\mu+1)^{2}-\nu^{2}}
 +\mathcal{O}\left(z^{\mu+3}\right),
\end{equation}
whereas from (\ref{16}) we compute that if $\nu = \mu +1$ then
\begin{multline}  \label{18}
 S_{\mu,\mu+1}^{(1)}(z) =
 \frac {z^{\mu+1}\ln(z) }{2\left(\mu+1 \right)} \\
 -\frac {z^{\mu+1}}{4(\mu+1)} \left\{ 
\frac {2\mu+1}{\mu(\mu+1) } +2\ln(2) -\gamma
 +\psi(\mu) +i\pi\right\} 
+\mathcal{O}\left(z^{\mu+3}\ln(z)\right),
\end{multline}
where $\psi(z)$ is the logarithmic derivative of the Gamma function, and $\gamma$ is Euler’s constant (see \cite[Eqs. 5.2.2 and 5.2.3]{NIST:DLMF}).

As another example we have, again from (\ref{16}),
\begin{multline}  \label{19}
S_{3,6}^{(1)}(z)=
-\frac{1}{20}z^4
+\frac{1}{240}z^6\ln(z) \\
+\frac{1}{240}
\left(\gamma - \ln(2)  - \frac{49}{40}-\frac{\pi i}{2}\right)z^6
+\mathcal{O}\left(z^{8}\ln(z)\right)
\quad (z \rightarrow 0).
\end{multline}

Next we can construct an explicit integral representation similar to (\ref{10}). To this end, using \cite[Eq. 10.5.3]{NIST:DLMF}
\begin{equation}  \label{20}
\mathscr{W}\left\{J_{\nu}(z),H_{\nu}^{(1)}(z)\right\}
=2i/(\pi z),
\end{equation}
we find from (\ref{02}), (\ref{46}) and (\ref{46a}), and variation of parameters that, for $\Re(\mu + \nu)>-1$,
\begin{equation}  \label{21}
 S_{\mu,\nu}^{(1)}(z)
 =\frac{\pi i}{2}
 \left[J_{\nu}(z)\int_{\infty \exp(\pi i/2)}^z 
 t^{\mu}H_{\nu}^{(1)}(t)dt  -H_{\nu}^{(1)}(z)\int_{0}^z 
 t^{\mu}J_{\nu}(t)dt \right].
\end{equation}

The second new Lommel function is defined similarly to (\ref{13}) and (\ref{14}) by
\begin{multline}  \label{22}
S_{\mu,\nu}^{(2)}(z)
=S_{\mu,\nu}(z)+iA(\mu,\nu)
\cos\left(\tfrac{1}{2}(\mu-\nu)\pi\right)H_{\nu}^{(2)}(z) \\
= S_{\mu,\nu}(z)+\frac{i\pi 2^{\mu-1}\Gamma\left(\tfrac{1}{2}\mu+\tfrac{1}{2}\nu
+\tfrac{1}{2}\right)}{\Gamma\left(\tfrac{1}{2}\nu
-\tfrac{1}{2}\mu+\tfrac{1}{2}\right)}H_{\nu}^{(2)}(z),
\end{multline}
and then from (\ref{06}) we have
\begin{equation}  \label{23}
S_{\mu,\nu}^{(2)}(z)=s_{\mu,\nu}(z)+ie^{(\nu-\mu)\pi i/2}A(\mu,\nu)J_{\nu}(z).
\end{equation}
This has the characterising behaviour of being subdominant at $z=0$ and $z=\infty \exp(-\pi i/2)$. Similarly to (\ref{21}) we can express it in the form
\begin{equation}  \label{24}
 S_{\mu,\nu}^{(2)}(z)
 =\frac{\pi i}{2}
 \left[H_{\nu}^{(2)}(z)\int_{0}^z 
 t^{\mu}J_{\nu}(t)dt
 -J_{\nu}(z)\int_{\infty \exp(-\pi i/2)}^z 
 t^{\mu}H_{\nu}^{(2)}(t)dt
 \right],
\end{equation}
for $\Re(\mu + \nu)>-1$. Also note from (\ref{13}), (\ref{22}) and \cite[Eq. 10.11.9]{NIST:DLMF} that $\overline{S_{\mu,\nu}^{(1)}(\bar{z})}=S_{\mu,\nu}^{(2)}(z)$.

Our third new Lommel function is simply defined by
\begin{equation}  \label{34}
 S_{\mu,\nu}^{(0)}(z)=\tfrac{1}{2}
 \left\{S_{\mu,\nu}^{(1)}(z)+S_{\mu,\nu}^{(2)}(z)\right\}.
\end{equation}
We note that from (\ref{03}), (\ref{04}), (\ref{15}), (\ref{23}) and (\ref{34}) that for $j=0,1,2$
\begin{equation}  \label{34a}
S_{\mu,\nu}^{(j)}(z) \sim \frac{z^{\mu+1}}{(\mu+1)^{2}-\nu^{2}}
\quad (z \rightarrow 0, \, \Re(\nu) > \Re(\mu)+1),
\end{equation}
and therefore all three are subdominant particular solutions at $z=0$. 

Next, from (\ref{15}) and (\ref{23})
\begin{equation}  \label{35}
S_{\mu,\nu}^{(0)}(z)=s_{\mu,\nu}(z)+ A(\mu,\nu)\sin\left(\tfrac{1}{2}(\mu-\nu)\pi \right)J_{\nu}(z).
\end{equation}
Hence from (\ref{34}) $S_{\mu,\nu}^{(0)}(x) \in \mathbb{R}$ for $x>0$. Its advantage over $s_{\mu,\nu}(z)$ is that it is well-defined for the same large parameter values as its parents $S_{\mu,\nu}^{(1)}(z)$ and $S_{\mu,\nu}^{(2)}(z)$.

To obtain an integral representation we have from \cite[Eq. 10.22.44]{NIST:DLMF} the identity
\begin{equation}  \label{36}
\int_{0}^{\infty} t^{\mu}Y_{\nu}(t)dt
=\frac{2}{\pi}A(\mu,\nu)\sin\left(\tfrac{1}{2}(\mu-\nu)\pi \right)
\quad (\Re(\mu \pm \nu)>-1,\, \Re(\mu)<\tfrac{1}{2}).
\end{equation}
Therefore from (\ref{05}) and (\ref{35})
\begin{multline}  \label{37}
 S_{\mu,\nu}^{(0)}(z)
 =\frac{\pi}{2}
 \left[J_{\nu}(z)\int_{z}^{\infty}
 t^{\mu}Y_{\nu}(t)dt 
+ Y_{\nu}(z)\int_{0}^z 
 t^{\mu}J_{\nu}(t)dt \right] \\
\quad (\Re(\mu + \nu)>-1,\, \Re(\mu)<\tfrac{1}{2}).
\end{multline}

Our new functions do not satisfy the reflection formula (\ref{07a}), but instead, from that relation along with (\ref{07}), (\ref{13}), (\ref{07a}) and \cite[Eq. 10.4.6]{NIST:DLMF}, we obtain
\begin{equation}  \label{25}
S_{\mu,-\nu}^{(1)}(z)
=S_{\mu,\nu}^{(1)}(z)+e^{(\mu+\nu)\pi i/2}
\sin(\pi \nu)A(\mu,\nu)H_{\nu}^{(1)}(z),
\end{equation}
and similarly
\begin{equation}  \label{26}
S_{\mu,-\nu}^{(2)}(z)
=S_{\mu,\nu}^{(2)}(z)+e^{-(\mu+\nu)\pi i/2}
\sin(\pi \nu)A(\mu,\nu)H_{\nu}^{(2)}(z).
\end{equation}
From (\ref{34}), (\ref{25}) and (\ref{26}) we also have
\begin{multline}  \label{38}
S_{\mu,-\nu}^{(0)}(z)
=S_{\mu,\nu}^{(0)}(z)+
\sin(\pi \nu)A(\mu,\nu)  \\
\times \left\{\cos\left(\tfrac{1}{2}(\mu+\nu)\pi\right)J_{\nu}(z)
-\sin\left(\tfrac{1}{2}(\mu+\nu)\pi\right)Y_{\nu}(z)\right\}.
\end{multline}

Analytic continuation formulas are straightforward to find. Firstly from (\ref{15}) and \cite[Eq. 10.11.1]{NIST:DLMF} ($m \in \mathbb{Z}$)
\begin{equation}  \label{27}
J_{\nu}\left(ze^{m\pi i}\right)
=e^{m\nu\pi i}J_{\nu}(z),
\end{equation}
and
\begin{equation}  \label{28}
s_{\mu,\nu}\left(ze^{m\pi i}\right)
=e^{m (\mu+1) \pi i}s_{\mu,\nu}(z),
\end{equation}
we get
\begin{multline}  \label{29}
S_{\mu,\nu}^{(1)}\left(ze^{m\pi i}\right)
=e^{m (\mu +1)\pi i}S_{\mu,\nu}^{(1)}(z) \\
+e^{(\mu-\nu+1)\pi i/2}
\left\{e^{m (\mu +1)\pi i} 
-e^{m \nu\pi i}\right\}A(\mu,\nu)J_{\nu}(z),
\end{multline}
and likewise
\begin{multline}  \label{30}
S_{\mu,\nu}^{(2)}\left(ze^{m\pi i}\right)
=e^{m (\mu +1)\pi i}S_{\mu,\nu}^{(2)}(z) \\
+e^{-(\mu-\nu+1)\pi i/2}
\left\{e^{m (\mu +1)\pi i} 
-e^{m \nu\pi i}\right\}A(\mu,\nu)J_{\nu}(z).
\end{multline}
As special case, using (\ref{27}) and
\begin{equation}  \label{31}
S_{\mu,\nu}^{(1)}(z)
=S_{\mu,\nu}^{(2)}(z)-2iA(\mu,\nu)
\cos\left(\tfrac{1}{2}(\mu-\nu)\pi\right)J_{\nu}(z),
\end{equation}
which comes from (\ref{13}) and (\ref{22}), we get
\begin{equation}  \label{32}
S_{\mu,\nu}^{(1)}(ze^{\pi i})
=e^{(\mu +1)\pi i}S_{\mu,\nu}^{(2)}(z),
\end{equation}
and on replacing $z$ by $ze^{-\pi i}$
\begin{equation}  \label{33}
S_{\mu,\nu}^{(2)}(ze^{-\pi i})
=e^{-(\mu +1)\pi i}S_{\mu,\nu}^{(1)}(z).
\end{equation}

From (\ref{29}) and (\ref{30}) for $m \in \mathbb{Z}$
\begin{multline}  \label{39}
S_{\mu,\nu}^{(0)}(ze^{m\pi i})
=e^{m (\mu +1)\pi i}S_{\mu,\nu}^{(0)}(z) \\
+\sin\left(\tfrac{1}{2}(\mu-\nu)\pi\right)
\left\{e^{m \nu\pi i}-e^{m (\mu +1)\pi i} \right\}
A(\mu,\nu)J_{\nu}(z).
\end{multline}

Finally we note that an expression for $S_{\mu,\nu}(ze^{m\pi i})$ comes from this, \cite[Eq. 10.11.2]{NIST:DLMF} and
\begin{equation}  \label{40}
S_{\mu,\nu}(z) = S_{\mu,\nu}^{(0)}(z)
- \frac{\pi 2^{\mu-1}\Gamma\left(\tfrac{1}{2}\mu+\tfrac{1}{2}\nu
+\tfrac{1}{2}\right)}{\Gamma\left(\tfrac{1}{2}\nu
-\tfrac{1}{2}\mu+\tfrac{1}{2}\right)}Y_{\nu}(z),
\end{equation}
which itself follows from (\ref{14}), (\ref{22}) and (\ref{34}).

\section{Asymptotic expansions for Lommel functions} 
\label{sec3}
Under the transformation $w=\nu^{-\mu-(3/2)} z^{1/2}y$, followed by $z \rightarrow \nu z$, (\ref{02}) then becomes
\begin{equation}  \label{61}
\frac{d^{2}w}{dz^{2}}+\left\{ \nu^{2}\left ( \frac{z^{2}-1}{z^{2}} \right)+\frac{1}{4z^{2}} \right \}w=z^{\mu-\frac{1}{2}},
\end{equation}
which has homogeneous solutions $z^{1/2}\mathscr{C}_{\nu}(\nu z)$ where $\mathscr{C}= J$, $Y$ or $H^{(j)}$ ($j=1,2$), and particular solutions $\nu^{-\mu-1} z^{1/2}\mathscr{S}_{\mu,\nu}(\nu z)$ where $\mathscr{S}=S$, $s$, or $S^{(j)}$ ($j=0,1,2$). Note the turning points at $z=\pm 1$. 

Uniform asymptotic expansions for these solutions involve Airy functions and Scorer functions, described below, and also the variables $\xi$ and $\zeta$ defined by \cite[Chap. 11, Eq. (10.05)]{Olver:1997:ASF}
\begin{equation}  \label{62}
\xi=\frac{2}{3}\zeta^{3/2}=\ln\left \{ \frac{1+\left ( 1-z^{2} \right )^{1/2}}{z} \right \}-\left ( 1-z^{2} \right )^{1/2}.
\end{equation}
The branches for $\zeta$ are such that $\zeta \in [0,\infty)$ when $z \in (0,1]$, and is continuous elsewhere in the cut plane $|\arg(z)|<\pi$. Thus $\zeta=+\infty,0,-\infty$ correspond to $z=0,1,+\infty$, respectively. In particular, the turning point at $z=1$ is mapped to a turning point at $\zeta=0$. Furthermore $\zeta $ is an analytic function of $z$ in the cut plane $|\arg(z)| < \pi$, in particular at the turning point $z=1$.

Unlike $\zeta$ the variable $\xi $ has a branch point at $z=1$. Any branch can be chosen in our subsequent expansions, as long as it is used consistently throughout the approximation in question. For a detailed description of the $z - \zeta$ transformation see \cite[Chap. 11, Sect. 10]{Olver:1997:ASF}.

\begin{figure}[hthp]
  \centering
  \includegraphics[trim={0 160 0 160},width=0.9\textwidth,keepaspectratio]{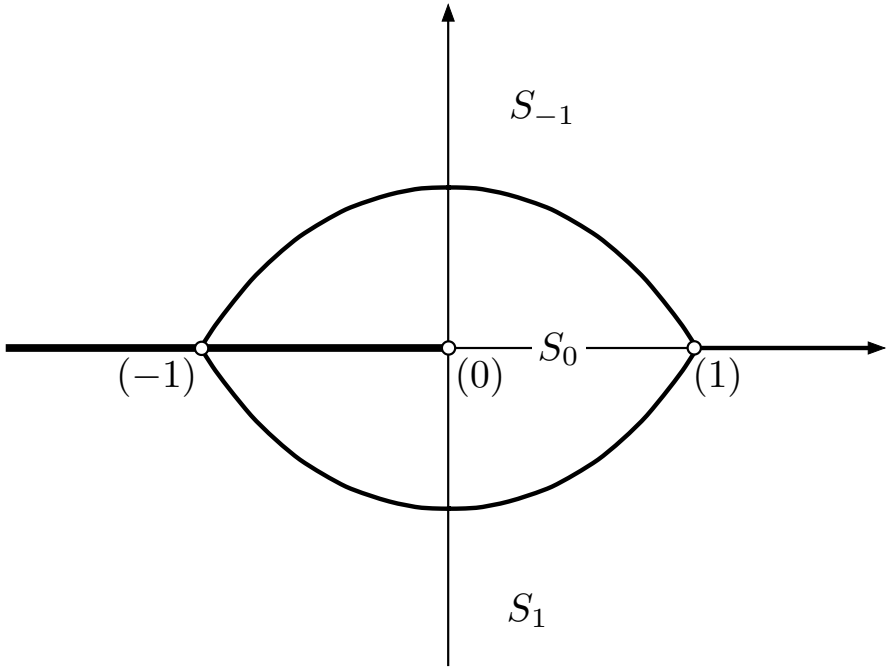}
  \caption{Regions $S_{j}$ ($j=0,\pm 1$)}
  \label{fig:fig1}
\end{figure}

Three regions $S_{j}$ ($j=0,\pm 1$) are depicted in \cref{fig:fig1}. The boundaries are the level curves $\Re(\xi)=0$ emanating from $z=\pm 1$. We also indicate a branch cut along $-\infty<z\leq 0$ associated with the multi-valued functions we are dealing with. The bounded eye-shaped region $S_{0}$ is the one in which $J_{\nu}(\nu z)$ is bounded as $\nu \rightarrow \infty$, and likewise $H_{\nu}^{(1)}(\nu z)$ (respectively $H_{\nu}^{(1)}(\nu z)$) is bounded in $S_{-1}$ (respectively $S_{1}$). All three Bessel functions are unbounded outside these regions. The subscripts $0,\pm 1$ correspond the the Airy function used to approximate each of these (see (\ref{79}) - (\ref{82}) below).

We next define $S(\delta)$ to be the the cut principal plane $-\pi < \arg(z) \leq \pi$ with all points removed that are at a distance less than $\delta$ from the interval $-\infty < z \leq -1$. Many of our approximations will be uniformly valid in this region. \textit{If a solution is unbounded at $z=0$ we tacitly assume this point must be excluded from the region of validity, without explicitly stating this}. Whether the solutions in question are bounded or not at $z=0$, we emphasise that all expansions for $z \in S(\delta)$ are uniformly valid arbitrarily close to this singularity. The same applies to $S^{(-1,0)}(\delta)$ and $S^{(-1,0)}(\delta)$ defined later. 

The region $S(\delta)$ is shown in \cref{fig:fig2}, along with several level curves $\Re(\xi)=\mathrm{constant}$. These level curves are crucial in determining regions of validity. This region is where the expansions involving Airy and Scorer are valid, which approximate homogeneous and inhomogeneous asymptotic solutions at the turning point $z=1$. 

For a point $z$ to be in $S(\delta)$ it is necessary and sufficient that there exist so-called progressive paths that link it to at least two of the three singularities at $z=0$ and $z=c \pm i \infty$ ($c$ any real constant). Each path must consist of a union of $R_{2}$ arcs (see \cite[Chap. 5, Sect. 3.3]{Olver:1997:ASF}), and as a point $t$ (say) moves along this path from one singularity to the other $\Re(\xi(t))$ must vary continuously and monotonically. We also require that the path not get too close to $z=-1$, specifically $|t+1|\geq \delta$ for all points on the path. 

The branch of $\xi(z)$ at $z=1$ can be chosen appropriately so that the continuity requirement is met for all points in $S(\delta)$. For example, if $z$ lies on $(1,\infty)$ we can choose the branch cut for $\xi$ to be $(-\infty,1]$ and to be continuous elsewhere in this cut plane, and two suitable paths to be vertical lines from $z$ to $z \pm i \infty$. Although not required, we can also choose a progressive path linking $z$ to $0$ by simply taking it to be the real line segment $[0,z]$. In all three cases $\Re(\xi)$ increases or decreases monotonically along the $R_{2}$ path, depending on which branches are taken in (\ref{62}); the actual choice does not matter. Thus all points lying on $(1,\infty)$ certainly lie in $S(\delta)$.

As another example, suppose $z$ lies on the cut $(-\infty,0]$ with $\arg(z)=\pi$. If $|z| < 1-\delta$ we can again take the branch cut for $\xi$ to be $(-\infty,1]$ and choose either a horizontal path linking the point to $0$, or a vertical path linking the point to $z + i \infty$. In addition, a suitable path linking it to $\infty$ in the lower half plane can be constructed by starting at $z$ and following the oval-like level curve it lies on, in a negative direction around the origin, until we hit the negative imaginary axis, and then from there follow the imaginary axis to $-i \infty$. For all three paths $\xi$ varies monotonically, and hence $z$ lies in $S(\delta)$.

On the other hand, if $\arg(z)=\pi$ but now $|z| \geq 1-\delta$, then by examining the level curves in \cref{fig:fig2} we see that any progressive path linking $z$ to $0$, or to infinity in the lower half plane, must pass through $z=-1$ in order to maintain the monotonicity requirement. Hence these paths violate our requirement of maintaining a minimum distance of $\delta$ to $z=-1$, and thus all such points must be excluded (and similarly those just below the cut when $|z| \geq 1-\delta$). Indeed all points sufficiently close to the interval $(-\infty,-1]$ must be excluded for the same reason.

However, points on the boundary of the shadow region in \cref{fig:fig2} meet the progressive path requirement, because it is readily seen that suitable progressive paths can be chosen linking these points to $0$, $z+i \infty$ (and also $z-i \infty$) in such a way that its distance from $z=-1$ is greater than or equal to $\delta$ along the path. This explains our requirement that all points in $S(\delta)$ are no closer than a distance $\delta$ from the interval $(-\infty,-1]$. This requirement could be relaxed somewhat, using certain level curves as boundaries in place of our horizontal lines, but this would be an unnecessary complication.

\begin{figure}[hthp]
  \centering
  \includegraphics[trim={20 160 20 160},width=1.0\textwidth,keepaspectratio]{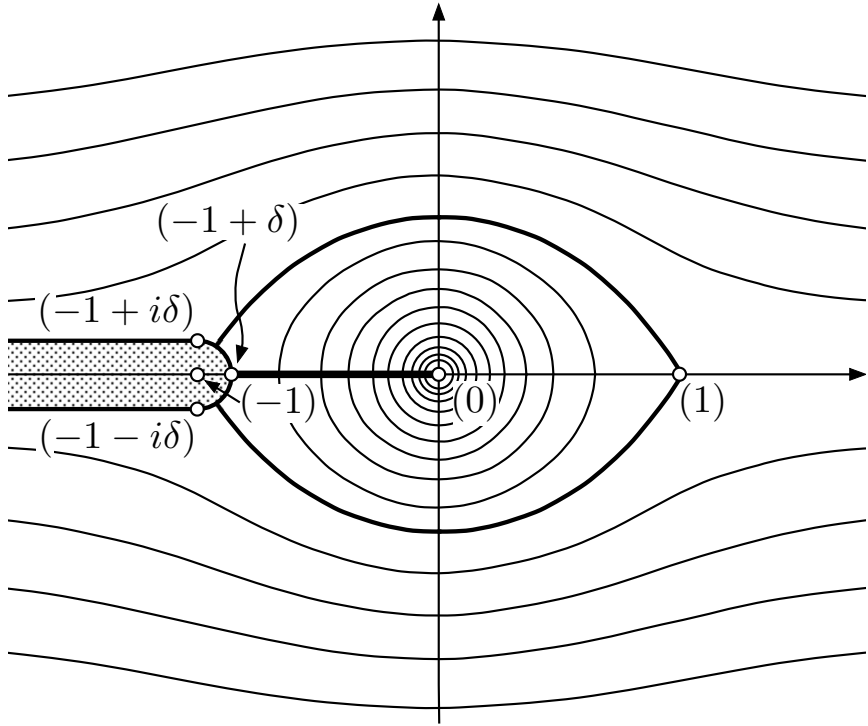}
  \caption{Region $S(\delta)$ and level curves $\Re(\xi)=\mathrm{constant}$}
  \label{fig:fig2}
\end{figure}

For $(j,k)=(-1,0)$, $(0,1)$ and $(-1,1)$ we define $S^{(j,k)}(\delta) \subset S(\delta)$ as the region consisting of points in $S_{j} \cup S_{k}$  whose distance from $S_{l}$ ($l\neq j,k$) is greater than or equal to $\delta$. The region $S^{(-1,0)}(\delta)$ is the unshaded region depicted in \cref{fig:fig2}, with $S^{(0,1)}(\delta)$ being the conjugate of this. The region $S^{(-1,1)}(\delta)$ is the unshaded region depicted in \cref{fig:fig3}. Note that both the turning points $z=\pm 1$ are excluded from all three of these. 

Each $S^{(j,k)}(\delta)$ is a region of validity for an asymptotic particular solution that is subdominant in this region. Their expansions are simpler since they are not valid at the turning point $z=1$ (and also $z=-1$), and consequently only require elementary functions. The regions are defined similarly to $S(\delta)$ in terms of progressive paths. The difference is that $S^{(j,k)}(\delta)$ is a subset of $S_{j} \cup S_{k}$, and the progressive path passing through a point in the region must keep a distance $\delta$ from \textit{both} turning points $\pm 1$. Also, the end points are pre-determined since they must be the two singularities in $S_{j} \cup S_{k}$.

\begin{figure}[hthp]
  \centering
  \includegraphics[trim={0 160 0 160},width=0.9\textwidth,keepaspectratio]{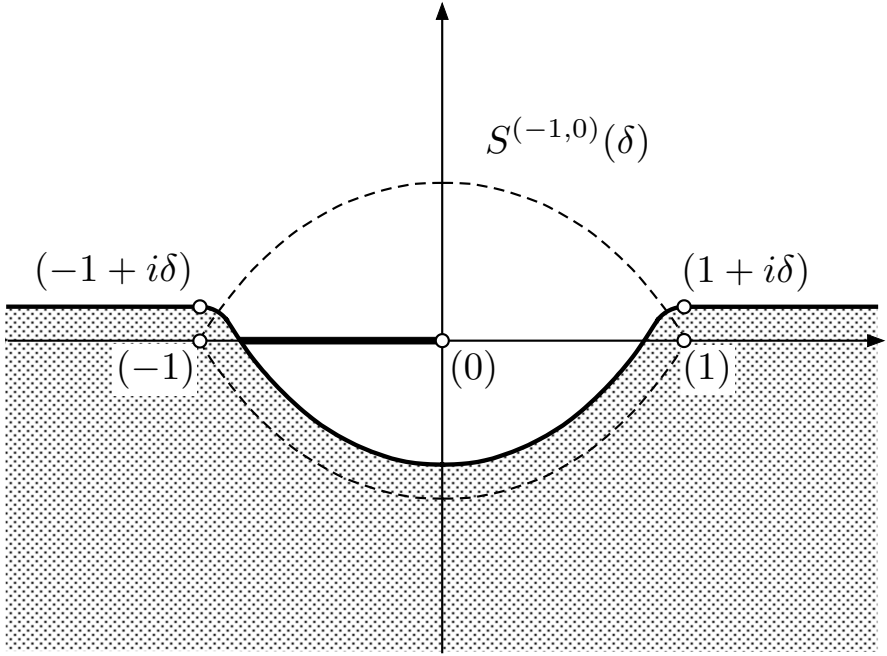}
  \caption{Region $S^{(-1,0)}(\delta)$}
  \label{fig:fig3}
\end{figure}

\begin{figure}[hthp]
  \centering
  \includegraphics[trim={0 160 0 160},width=0.9\textwidth,keepaspectratio]{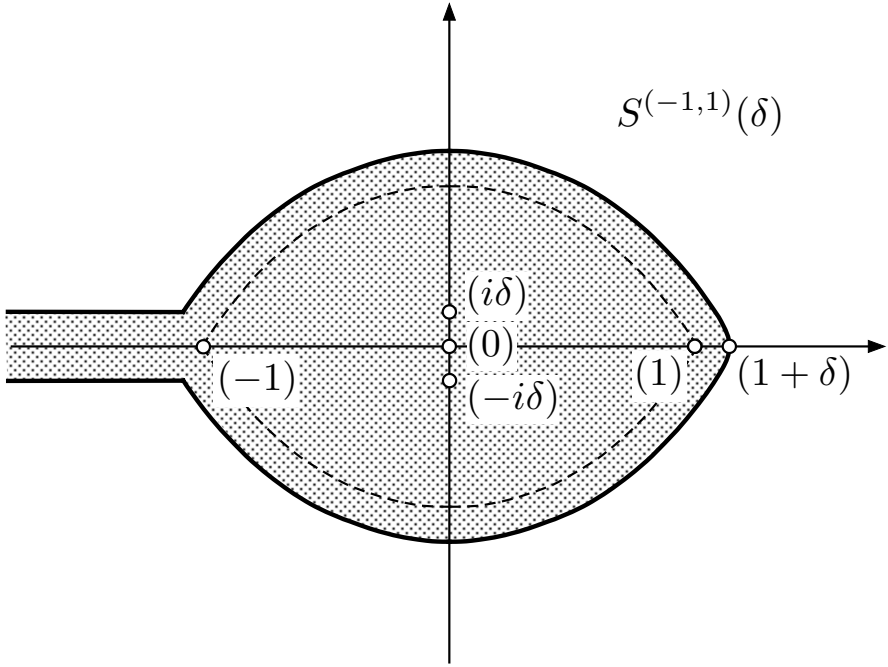}
  \caption{Region $S^{(-1,1)}(\delta)$}
  \label{fig:fig4}
\end{figure}

Our asymptotic expansions valid at $z=1$ involve solutions of the inhomogeneous Airy equation
\begin{equation}  \label{63}
\frac{d^{2}w}{dz^{2}}-zw=1.
\end{equation}
This has homogeneous solutions the Airy functions 
\begin{equation}  \label{63a}
\mathrm{Ai}_{l}(z)=\mathrm{Ai}(z e^{-2\pi i l/3})
\quad (l=0,\pm 1).
\end{equation}
For large $\nu$ each $\mathrm{Ai}_{l}(\nu^{2/3}\zeta)$ ($l=0,\pm 1$) is bounded in $S_{l}$ and is unbounded outside this region.

Particular solutions of (\ref{63}) are
\begin{equation}  \label{64}
 \mathrm{Wi}^{(-1,1)}(z)=\pi \mathrm{Hi}(z),
\end{equation}
 \begin{equation}  \label{65}
 \mathrm{Wi}^{(0,1)}(z)=\pi e^{-2\pi i/3}\mathrm{Hi}\left(ze^{-2\pi i/3}\right),
\end{equation}
and
\begin{equation}  \label{66}
 \mathrm{Wi}^{(-1,0)}(z)=\pi e^{2\pi i/3}\mathrm{Hi}
 \left(ze^{2\pi i/3} 
\right),
\end{equation}
where $\mathrm{Hi}(z)$ is the Scorer function defined by \cite[Sect. 9.12]{NIST:DLMF}
\begin{equation}  \label{67}
\mathrm{Hi}(z)=\frac{1}{\pi }\int_{0}^{\infty} \exp  \left(-\tfrac{1}{3} t^{3}+z t \right)dt.
\end{equation}

$\mathrm{Wi}^{(-1,1)}(z)$ is the uniquely defined particular solution of the inhomogeneous Airy equation (\ref{63}) having the behaviour
\begin{equation}  \label{68}
\mathrm{Wi}^{(-1,1)}(z) \sim -1/z \quad \left(z\to \infty,\; \left| \arg(-z) \right|\le \tfrac{2}{3}\pi -\delta \right).
\end{equation}
This function grows exponentially as $z \rightarrow \infty$ in $|\arg(z)|\le \tfrac{1}{3}\pi -\delta$ (see \cite[Eq. 9.12.29]{NIST:DLMF}). Similarly of course for $\mathrm{Wi}^{(0,1)}(z)$ and $\mathrm{Wi}^{(-1,0)}(z)$ in rotated sectors. From this one can show for $(j,k)=(-1,0)$, $(0,1)$ and $(-1,1)$, and $\zeta$ given by (\ref{62}), that $\mathrm{Wi}^{(j,k)}(\nu^{2/3}\zeta)$ has the fundamental property of being bounded in $S_{j} \cup S_{k}$ and unbounded in $\mathbb{C} \setminus \left\{S_{j} \cup S_{k}\right\}$ as $\nu \rightarrow \infty$. We shall shortly see that these play an important role in approximating particular solutions of (\ref{02}) that are subdominant in the same regions, namely the Lommel functions $S_{\mu,\nu}(z)$, $S_{\mu,\nu}^{(1)}(z)$ and $S_{\mu,\nu}^{(2)}(z)$.

Let us now set up the coefficients that will appear in our expansions. These require repeated integration, and to facilitate these we introduce the variable
\begin{equation}  \label{69}
\beta=\frac{1}{\sqrt{1-z^{2}}}.
\end{equation}
Here the branch of the square root is positive for $-1<z<1$ and is continuous in the plane having a cuts along $(-\infty,-1]$ and $[1,\infty)$. Note that $\beta \rightarrow 0$ as $z \rightarrow \infty$. From (\ref{69}) and \cite[Eqs. (5.6) -- (5.8)]{Dunster:2020:SEB} (in a slightly modified form) we then define a sequence of polynomials via
\begin{equation}  \label{70}
\mathrm{E}_{1}(\beta)=\tfrac{1}{24}\beta
\left(5\beta^{2}-3\right),
\end{equation}
\begin{equation}  \label{71}
\mathrm{E}_{2}(\beta)=
\tfrac{1}{16}\beta^{2}\left(\beta^{2}-1\right)
\left(5\beta^{2}-1\right),
\end{equation}
and for $s=2,3,4\cdots$
\begin{equation}  \label{72}
\mathrm{E}_{s+1}(\beta) =
\frac{1}{2} \beta^{2} \left(\beta^{2}-1 \right)\mathrm{E}_{s}^{\prime}(\beta)
+\frac{1}{2}\int_{0}^{\beta}
p^{2}\left(p^{2}-1 \right)
\sum\limits_{j=1}^{s-1}
\mathrm{E}_{j}^{\prime}(p)
\mathrm{E}_{s-j}^{\prime}(p) dp.
\end{equation}

We next define coefficient functions from \cite[Eqs. (1.16) -- (1.18)]{Dunster:2020:SEB}. These involve two sequences of numbers given by $\left\{a_{s}\right\}_{s=1}^{\infty}$ and $\left\{\tilde{a}_{s}\right\}_{s=1}^{\infty}$, where $a_{1}=a_{2}=\frac{5}{72}$, $\tilde{a}_{1}=\tilde{a}_{2}=-\frac{7}{72}$, and subsequent terms $a_{s}$ and $\tilde{{a}}_{s}$ ($s=3,4,5,\cdots $) satisfying the same recursion formula
\begin{equation}  \label{73}
b_{s+1}=\frac{1}{2}\left(s+1\right) b_{s}+\frac{1}{2}
\sum\limits_{j=1}^{s-1}{b_{j}b_{s-j}} 
\quad (s=2,3,4,\cdots).
\end{equation}

Next for $s=1,2,3,\cdots$ define
\begin{equation}  \label{74}
\mathcal{E}_{s}(z) =\mathrm{E}_{s}(\beta) +
(-1)^{s}a_{s}s^{-1}\xi^{-s},
\end{equation}
and
\begin{equation}  \label{75}
\tilde{\mathcal{E}}_{s}(z) =\mathrm{E}_{s}(\beta)
+(-1)^{s}\tilde{a}_{s}s^{-1}\xi^{-s}.
\end{equation}
From \cite[Thm. 3.4]{Dunster:2020:SEB} homogeneous asymptotic solutions of (\ref{61}) are then given by $w(\nu,z)=z^{1/2}w_{l}(\nu,z)$, where
\begin{equation}  \label{79}
w_{l}(\nu,z)=\mathrm{Ai}_{l} \left(\nu^{2/3}\zeta \right)\mathcal{A}(\nu,z)+\mathrm{Ai}'_{l} \left(\nu^{2/3}\zeta \right)\mathcal{B}(\nu,z)\quad (l=0,\pm 1),
\end{equation}
in which the Airy functions of complex argument are defined by (\ref{63a}).

The coefficient functions $\mathcal{A}(\nu,z)$ and $\mathcal{B}(\nu,z)$ are slowly varying in $\nu$ and are analytic in the cut plane $\mathbb{C} \setminus (-\infty,-1]$. As $\nu \rightarrow \infty$, uniformly for $z \in S(\delta) \setminus\left\{1\right\}$ they possess the asymptotic expansions
\begin{equation}  \label{76}
\mathcal{A}(\nu,z) \sim \phi(z)\exp \left\{ \sum\limits_{s=1}^{\infty}\frac{
\tilde{\mathcal{E}}_{2s}(z) }{\nu^{2s}}\right\} \cosh \left\{ \sum\limits_{s=0}^{\infty}\frac{\tilde{\mathcal{E}}_{2s+1}(z) }{\nu^{2s+1}}\right\},
\end{equation}
and
\begin{equation}  \label{77}
\mathcal{B}(\nu,z) \sim \frac{\phi(z)}{\nu^{1/3}\sqrt{\zeta}}\exp \left\{ \sum\limits_{s=1}^{\infty}\frac{
\mathcal{E}_{2s}(z) }{\nu^{2s}}\right\} \sinh \left\{ \sum\limits_{s=0}^{\infty}\frac{\mathcal{E}_{2s+1}(z) }{\nu^{2s+1}}\right\},
\end{equation}
where
\begin{equation}  \label{78}
\phi(z)=\left(\frac{\zeta}{1-z^2}\right)^{1/4}.
\end{equation}
Note that $\phi(z)$ has a removable singularity at the turning point $z=1$ since $\zeta$, as defined by (\ref{62}), has a simple zero at this point.

Error bounds are provided in \cite[Thms. 3.4 and 4.2]{Dunster:2020:SEB}. Note here (for later convenience) our notation is slightly different in that we have omitted a factor $z^{1/2}$ in (\ref{78}), so that the asymptotic solutions are $z^{1/2}w_{l}(\nu,z)$ rather than $w_{l}(\nu,z)$. We also do not have the truncation integer $m$ dependence on $\mathcal{A}$ and $\mathcal{B}$ since we will not include the error terms here. We should emphasise that all expansions in this paper are rigorously established by the error bounds of the general results that we use from \cite{Dunster:2020:ASI} and \cite{Dunster:2020:SEB}.

As shown in \cite{Dunster:2017:COA} the expansions (\ref{76}) and (\ref{77}) are very accurate and easy to compute. We remark that $\mathcal{E}_{2s}(z)$ and $\tilde{\mathcal{E}}_{2s}(z)$ are not analytic at  $z=1$, thus the expansions cannot be used directly at or near this point. However both $\mathcal{A}(\nu,z)$ and $\mathcal{B}(\nu,z)$ are themselves analytic at $z=1$, and there are two ways to compute them near, and at, this turning point. Firstly, if many terms are required for high accuracy, Cauchy's integral formula can be used, as given in \cite[Thm. 4.2]{Dunster:2020:SEB} and \cite[Eqs. (19) and (20)]{Dunster:2020:ASI}. In this the expansions (\ref{76}) and (\ref{77}) for $\mathcal{A}(\nu,z)$ and $\mathcal{B}(\nu,z)$ are inserting into the Cauchy integral representation of these functions, and since they are valid on the contour of integration that encloses $z=1$ this allows accurate and stable computation of each function.

On the other hand, if only a few terms in an expansion are required then (\ref{76}) and (\ref{77}) can be expanded in a traditional asymptotic series involving inverse powers of $\nu^{2}$. We find that each coefficient has a removable singularity at $z=1$ and can be computed via a Taylor series. The same is true for the expansions for inhomogeneous solutions that involve the Scorer functions (see (\ref{93}) - (\ref{95}) below). \textit{From now on we shall state that these expansions are uniformly valid for $z \in S(\delta)$ with the understanding that if $z$ is close to $1$ then Cauchy's integral formula or a re-expansion, as described above, can be used to evaluate the coefficient functions}.

The following theorem \cite[Thm. 3.1]{Dunster:2021:NKF} generalises the fact that coefficients in such a re-expansion have a removable singularity at a singularity, and the resulting expansions are uniformly valid in a neighbourhood of that point. We state it here since we shall use this later.
\begin{theorem} \label{thm:nopoles}
Let $0<\rho_{1}<\rho_{2}$, $\nu>0$, $z_{0}\in \mathbb{C}$, $H(\nu,z)$ be an analytic function of $z$ in the open disk $D=\{z:\, |z-z_{0}|<\rho_{2}\}$, and $h_{s}(z)$ ($s=0,1,2,\cdots$) be a sequence of functions that are analytic in $D$ except possibly for an isolated singularity at $z=z_{0}$. If $H(\nu,z)$ is known to possess the asymptotic expansion
\begin{equation}  \label{109}
H(\nu,z) \sim \sum\limits_{s=0}^{\infty}\frac{h_{s}(z)}{\nu^{s}}
\quad (\nu \rightarrow \infty),
\end{equation}
in the annulus $\rho_{1}<|z-z_{0}|<\rho_{2}$, then $z_{0}$ is an ordinary point or a removable singularity for each $h_{s}(z)$, and the expansion (\ref{109}) actually holds for all $z \in D$ (with $h_{s}(z_{0})$ defined by the limit of this function at $z_{0}$ if it is a removable singularity).
\end{theorem}

From (\ref{46}), (\ref{46a}), (\ref{46b}) and (\ref{79}) we identify homogeneous solutions of (\ref{61}) that are recessive at $z=0$ and $z=\infty$ for $0<\arg(z)<\pi$ and $-\pi<\arg(z)<0$, and this yields
\begin{equation}  \label{80}
J_\nu(\nu z) \sim \sqrt{2} \nu^{-1/3}w_{0}(\nu,z),
\end{equation}
\begin{equation}  \label{81} 
H_\nu^{(1)}(\nu z) \sim 
2^{3/2} e^{-\pi i/3} \nu^{-1/3} w_{-1}(\nu,z),
\end{equation}
and
\begin{equation}  \label{82} 
H_\nu^{(2)}(\nu z) \sim 
2^{3/2} e^{\pi i/3} \nu^{-1/3} w_{1}(\nu,z).
\end{equation}
The proportionality constants follow from matching the solutions at $z=\infty$ (see \cite[Chap. 11]{Olver:1997:ASF}. We shall use these expansions in what follows.

Let us now turn to the main part of the paper, namely inhomogeneous asymptotic solutions of (\ref{61}). From \cite[Thm. 4]{Dunster:2020:ASI} these involve coefficients defined by
\begin{equation}  \label{83}
G_{\mu,0}(z)=
\frac{z^{\mu+\frac{3}{2}}}{z^{2}-1},
\end{equation}
and for $s=0,1,2,\cdots$
\begin{equation}  \label{84}
G_{\mu,s+1}(z)
=\frac{G_{\mu,s}(z)+4z^2G_{\mu,s}''(z)}{4\left(1-z^{2}\right)}.
\end{equation}
Here we introduce the subscript to signify $\mu$ dependence. From the theorem we have for $(j,k)=(-1,0)$, $(0,1)$ and $(-1,1)$ three fundamental asymptotic particular solutions given by
\begin{equation}  \label{85}
w_{\mu}^{(j,k)}(\nu,z) \sim \frac{1}{\nu^{2}}\sum\limits_{s=0}^{\infty} \frac{G_{\mu,s}(z)}{\nu^{2s}},
\end{equation}
as $\nu \rightarrow \infty$, $\mu \in \mathbb{C}$ bounded, uniformly for $z \in S^{(j,k)}(\delta)$. If $\Re(\mu)<-1$ then the singularity $z=0$ must be excluded for $(j,k)=(-1,0)$ and $(0,1)$, but the expansions remain uniformly valid arbitrarily close to this point. This is true throughout this section, in particular for (\ref{95}) below.

All these assertions come from the error bounds provided by \cite[Eqs. (52) and (53)]{Dunster:2020:ASI} with $\hat{G}_{s}(z)=G_{\mu,s}(z)$, $f(z)=(1-z^2)/z^2$ and
\begin{equation}  \label{86}
\Phi(z)=\frac {z^{2} \left(z^{2}+4 \right) }
{4\left( z^2-1 \right)^{3}}.
\end{equation}
In the event of $\Re(\mu)<-\frac{3}{2}$ all particular solutions $w_{\mu}^{(j,k)}(\nu,z)$ are unbounded at $z=0$, and the integrals in the error bounds diverge at $z=0$. However, in this case \cite[Thm. 3]{Dunster:2020:ASI} can be used instead, verifying that the expansions (\ref{85}) remain valid uniformly for $z \in S^{(j,k)}(\delta)$ (with $z \neq 0$), as asserted.

If we truncate the series in (\ref{85}) and replace the upper limit by $n-1$, then from the error bounds the difference of the finite series and the exact solution is $\mathcal{O}(\nu^{-2n-2})$ as as $\nu \rightarrow \infty$, uniformly for $z \in S^{(j,k)}(\delta)$. Moreover as $z \rightarrow \infty$ this error term is $\mathcal{O}(z^{\mu-2n-(1/2)})$, again for $z \in S^{(j,k)}(\delta)$.

For $(j,k)= (-1,0)$ and $(0,1)$ this error term is also $\mathcal{O}(z^{\mu+(3/2)})$ as $z \rightarrow 0$. This cannot be deduced directly from the error bounds, but is implicitly true because each coefficient $G_{\mu,s}(z)$ is $\mathcal{O}(z^{\mu +(3/2)})$, and so is $w_{\mu}^{(j,k)}(\nu,z)$ for the aforementioned values of $(j,k)$ (see (\ref{88}) and (\ref{89}) below), and the error term is the difference of $w_{\mu}^{(j,k)}(\nu,z)$ and a finite sum of terms involving $G_{\mu,s}(z)$.

We now identify the unique doubly subdominant solutions.
\begin{theorem} \label{thm:SG}
As $\nu \rightarrow \infty$ with $\mu$ fixed, uniformly in the stated regions
\begin{equation}  \label{87}
S_{\mu,\nu}(\nu z)
\sim \frac{\nu^{\mu}}{z^{1/2}}\sum\limits_{s=0}^{\infty} \frac{G_{\mu,s}(z)}{\nu^{2s+1}}
\quad (z \in S^{(-1,1)}(\delta)),
\end{equation}
\begin{equation}  \label{88}
S_{\mu,\nu}^{(1)}(\nu z) 
\sim \frac{\nu^{\mu}}{z^{1/2}}\sum\limits_{s=0}^{\infty} \frac{G_{\mu,s}(z)}{\nu^{2s+1}}
\quad (z \in S^{(-1,0)}(\delta)),
\end{equation}
\begin{equation}  \label{89}
S_{\mu,\nu}^{(2)}(\nu z) 
\sim \frac{\nu^{\mu}}{z^{1/2}}\sum\limits_{s=0}^{\infty} \frac{G_{\mu,s}(z)}{\nu^{2s+1}}
\quad (z \in S^{(0,1)}(\delta)),
\end{equation}
and
\begin{equation}  \label{89aa}
S_{\mu,\nu}^{(0)}(\nu z) 
\sim \frac{\nu^{\mu}}{z^{1/2}}\sum\limits_{s=0}^{\infty} \frac{G_{\mu,s}(z)}{\nu^{2s+1}}
\quad (z \in S^{(-1,0)}(\delta) \cap S^{(0,1)}(\delta)).
\end{equation}
\end{theorem}

\begin{proof}
$\nu^{-\mu-1}z^{1/2}S_{\mu,\nu}(\nu z)$ is the unique solution of (\ref{61}) that is subdominant in $S_{-1} \cup S_{1}$. Thus it must be equal to $w_{\mu}^{(-1,1)}(\nu,z)$ which shares this property, and hence
\begin{equation}  \label{89a}
S_{\mu,\nu}(\nu z) = \nu^{\mu+1} z^{-1/2} w_{\mu}^{(-1,1)}(\nu,z).
\end{equation}
Therefore from (\ref{85}) we then get (\ref{87}). The expansions (\ref{88}) and (\ref{89}) follow similarly from the identifications
\begin{equation}  \label{89b}
S_{\mu,\nu}^{(1)}(\nu z) = \nu^{\mu+1} z^{-1/2} w_{\mu}^{(-1,0)}(\nu,z),
\end{equation}
\begin{equation}  \label{89c}
S_{\mu,\nu}^{(2)}(\nu z) = \nu^{\mu+1} z^{-1/2} w_{\mu}^{(0,1)}(\nu,z),
\end{equation}
and (\ref{89aa}) comes from (\ref{34}), (\ref{88}) and (\ref{89}).
\end{proof}

We now move on to obtaining expansions for these functions that are valid at the turning point $z=1$. We begin by establishing an expansion for a connection coefficient that is required. From \cite[Eq. (112)]{Dunster:2020:ASI}, with $u=\nu$, $m=\infty$, $w_{m,-1}(\nu,z)=z^{1/2}w_{-1}(\nu,z)$, and $\mu$ dependence indicated in our functions, we have the important connection formula
\begin{equation}  \label{90}
w_{\mu}^{(-1,0)}(\nu,z) = w_{\mu}^{(-1,1)}(\nu,z)-2\pi e^{\pi i/6}\gamma_{\mu}(\nu)z^{1/2}w_{-1}(\nu,z),
\end{equation}
where $w_{-1}(\nu,z)$ is given by (\ref{79}), $w_{\mu}^{(-1,0)}(\nu,z)$, $w_{\mu}^{(-1,1)}(\nu,z)$ are given by (\ref{85}), and the coefficient $\gamma_{\mu}(\nu)$ is the one in question that will be used in all our approximations valid at the turning point $z=1$. Although we do not obtain an explicit representation, we are able to obtain an asymptotic approximation for large $\nu$. Specifically, using (\ref{14}), (\ref{81}), (\ref{89a}) and (\ref{89b}) we obtain the required approximation
\begin{equation}  \label{91}
\gamma_{\mu}(\nu) \sim 
\frac{2^{\mu-\frac{1}{2}}
\Gamma\left(\tfrac{1}{2}\mu+\tfrac{1}{2}\nu+\tfrac{1}{2}\right)}
{\nu^{\mu+\frac{4}{3}}\Gamma\left(\tfrac{1}{2}\nu
-\tfrac{1}{2}\mu+\tfrac{1}{2}\right)}
\quad (\nu \rightarrow \infty).
\end{equation}

Although this is not exact, if we expand the RHS as an asymptotic expansion for large $\nu$ it agrees to all orders. So for example with the aid of Stirling's formula \cite[Eq. 5.11.3]{NIST:DLMF}
\begin{equation}  \label{92}
\gamma_{\mu}(\nu) =\frac{1}{\sqrt{2}\,\nu^{4/3}}
\left\{1+\frac {\mu\left( 1-\mu^{2}\right) }{6{\nu}^{2}}
+\mathcal{O}\left( \frac{1}{\nu^4} \right)
\right\}.
\end{equation}

With this constant having been determined asymptotically, a third coefficient function $\mathcal{G}_{\mu}(\nu,z)$, companion to (\ref{76}) and (\ref{77}) and analytic for the same values of $z$ as them, is then defined from \cite[Eqs. (117) and (120)]{Dunster:2020:ASI}. It is also slowly varying for large $\nu$, and specifically for $\nu \rightarrow \infty$, $\mu \in \mathbb{C}$ fixed, uniformly for $z \in S(\delta)$
\begin{equation}  \label{93}
\mathcal{G}_{\mu}(\nu,z) \sim \frac{1}{\nu^{2}}\sum\limits_{s=0}^{\infty} \frac{G_{\mu,s}(z)}{\nu^{2s}} -\frac{\gamma_{\mu}(\nu)z^{1/2}\phi(z)J(\nu,z)}{\nu^{2/3}\zeta},
\end{equation}
where $J(\nu,z)$ has the formal expansion under the same conditions
\begin{multline}  \label{94}
 J(\nu,z) \sim -\exp \left\{\sum\limits_{s=1}^{\infty} \frac{\tilde{\mathcal{E}}_{2s}(z)}{\nu^{2s}} \right\}\cosh \left\{\sum\limits_{s=0}^{\infty} 
\frac{\tilde{\mathcal{E}}_{2s+1}(z)}{\nu^{2s+1}}  
\right\}\sum\limits_{k=0}^{\infty} \frac{(3k)!}{k!\left( 3\nu^{2}\zeta^{3} \right)^{k}} \\ 
 +\frac{1}{\nu\zeta^{3/2}}\exp \left\{\sum\limits_{s=1}^{\infty} \frac{{\mathcal{E}}_{2s}(z)}{\nu^{2s}}  \right\}\sinh \left\{\sum\limits_{s=0}^{\infty} 
\frac{{\mathcal{E}}_{2s+1}(z)}{\nu^{2s+1}}  \right\}\sum\limits_{k=0}^{\infty}
\frac{(3k+1)!}{k!\left(3\nu^{2}\zeta^{3} \right)^{k}}.
\end{multline}

From \cite[Thm. 7]{Dunster:2020:ASI} the solutions (\ref{85}) have the uniform asymptotic expansions
\begin{multline}  \label{95}
w_{\mu}^{(j,k)}(\nu,z)=\gamma_{\mu}(\nu)\left\{\mathrm{Wi}^{(j,k)}\left(\nu^{2/3}\zeta  \right)\mathcal{A}(\nu,z) \right. \\ + \left. {\mathrm{Wi}^{(j,k)}}^{\prime}\left(\nu^{2/3}\zeta \right)\mathcal{B}(\nu,z) \right\} +\mathcal{G}_{\mu}(\nu,z)
 \quad (j,k=0,\pm1, \, j<k).
\end{multline}

We note that
 \begin{equation}  \label{95a}
 {\mathrm{Wi}^{(0,1)}}^{\prime}\left(\nu^{2/3}\zeta \right)
 =\pi e^{2\pi i/3}\mathrm{Hi}^{\prime}
 \left(\nu^{2/3}\zeta e^{-2\pi i/3}\right),
\end{equation}
and
\begin{equation}  \label{95b}
{\mathrm{Wi}^{(-1,0)}}^{\prime}\left(\nu^{2/3}\zeta \right)
=\pi e^{-2\pi i/3}\mathrm{Hi}^{\prime}
\left(\nu^{2/3}\zeta e^{2\pi i/3} \right).
\end{equation}

While the function $\mathcal{G}_{\mu}(\nu,z)$ is analytic at $z=1$, the terms in the series (\ref{93}) and (\ref{94}) for it are not, and so these cannot be directly used near this point. However, an asymptotic expansion for $\mathcal{G}_{\mu}(\nu,z)$ can be computed near this point in similar way to $\mathcal{A}(\nu,z)$ and $\mathcal{B}(\nu,z)$, as described above. That is, either by Cauchy's integral formula (see \cite[Eq. (121)]{Dunster:2020:ASI}), or by re-expanding the series in inverse powers of $\nu^{2}$ for a few terms. In the latter case the coefficients of such a re-expansion have a removable singularity at $z=1$ by virtue of \cref{thm:nopoles}. Thus as $\nu \rightarrow \infty$, with $\mu \in \mathbb{C}$ fixed, (\ref{95}) provides asymptotic expansions for particular solutions of (\ref{61}), that are uniformly for $z \in S(\delta)$, just like (\ref{79}) do for the homogeneous solutions.

Matching of these expansions with the Lommel functions follows immediately from (\ref{64}), (\ref{65}), (\ref{66}), (\ref{89a}), (\ref{89b}) and (\ref{89c}). In particular we have
\begin{theorem}
\begin{multline}  \label{96}
S_{\mu,\nu}(\nu z) = \nu^{\mu+1} z^{-1/2} 
\Big[\mathcal{G}_{\mu}(\nu,z) \\
\left. +\pi\gamma_{\mu}(\nu)\left\{\mathrm{Hi}\left(\nu^{2/3}\zeta  \right)\mathcal{A}(\nu,z)+
{\mathrm{Hi}}^{\prime}\left(\nu^{2/3}\zeta \right)
\mathcal{B}(\nu,z) \right\} \right],
\end{multline}
and
\begin{multline}  \label{98}
S_{\mu,\nu}^{(0)}(\nu z) = \nu^{\mu+1} z^{-1/2} 
\Big[\mathcal{G}_{\mu}(\nu,z) \\
\left. -\pi\gamma_{\mu}(\nu)\left\{\mathrm{Gi}\left(\nu^{2/3}\zeta  \right)\mathcal{A}(\nu,z)+
{\mathrm{Gi}}^{\prime}\left(\nu^{2/3}\zeta \right)
\mathcal{B}(\nu,z) \right\} \right],
\end{multline}
where $\mathrm{Gi}(z)$ is the Scorer function defined by \cite[Eq. 9.12.4]{NIST:DLMF}. Expansions for $\mathcal{A}(\nu,z)$, $\mathcal{B}(\nu,z)$ and $\mathcal{G}_{\mu}(\nu,z)$ for $\nu \rightarrow \infty$ with $\mu \in \mathbb{C}$ fixed, uniformly for $z \in S(\delta)$, are furnished by (\ref{76}), (\ref{77}) and (\ref{93}).
\end{theorem}

\begin{proof}
From (\ref{64}), (\ref{87}) and (\ref{95}) we arrive at (\ref{96}). Next from (\ref{65}), (\ref{66}) and \cite[Eq. 9.12.12]{NIST:DLMF}
\begin{equation}  \label{97}
\mathrm{Gi}(z)=-\frac{1}{2 \pi} \left\{
\mathrm{Wi}^{(0,1)}(z)+\mathrm{Wi}^{(-1,0)}(z)
\right\}.
\end{equation}
Thus from this along with (\ref{34}), (\ref{65}), (\ref{66}), (\ref{89b}), (\ref{89c}) and (\ref{95}) we get (\ref{98})
\end{proof}

The corresponding expansions for $S_{\mu,\nu}^{(j)}(\nu z)$ ($j=1,2$) are the same as (\ref{96}) except with $\mathrm{Hi}(\nu^{2/3}\zeta)$ replaced respectively by $$e^{\pm 2\pi i/3}\mathrm{Hi}(\nu^{2/3}\zeta e^{\pm 2\pi i/3}),$$ and likewise for the derivatives (see (\ref{95a}) and (\ref{95b})). An expansion for $s_{\mu,\nu}(\nu z)$ follows from (\ref{35}), (\ref{79}), (\ref{80}) and (\ref{98}).

\begin{remark} \label{remark1}
Near $z=0$ one should use the simpler expansion (\ref{89aa}) instead of (\ref{98}), since in the latter there is a cancellation of terms that are large relative to the function itself, rendering it not very numerically stable for very small $|z|$. The same is true that (\ref{88}) and (\ref{89}) should be used in place of their Scorer function expansions when $|z|$ is very small.
\end{remark}

\section{Anger-Weber functions and Neumann polynomials} \label{sec4}

Here we use the results of the previous section to obtain uniform asymptotic expansions for the Anger-Weber functions, and also the Neumann polynomials. Before doing so we record some relevant relations for the Anger-Weber functions. 

Firstly the Anger function is defined by
\begin{equation}  \label{41}
\mathbf{J}_{\nu}(z)=\frac{1}{\pi}\int_{0}^{\pi}\cos(\nu \theta-z\sin(\theta))d\theta,
\end{equation}
and is a solution of \cite[Sect. 11.10(ii)]{NIST:DLMF}
\begin{equation}  \label{42}
\frac{d^{2}y}{dz^{2}}+\frac{1}{z}\frac{dy}{dz}
+\left(1-\frac{\nu^{2}}{z^{2}} \right)y
=\frac{\mathbf{j}_{0}(\nu)}{z}+\frac{\mathbf{j}_{-1}(\nu)}{z^{2}},
\end{equation}
where
\begin{equation}  \label{43}
\mathbf{j}_{0}(\nu)=\frac{\sin(\pi \nu)}{\pi}, \quad
\mathbf{j}_{-1}(\nu)=-\frac{\nu \sin(\pi \nu)}{\pi}.
\end{equation}
In terms of the Lommel function $s_{\mu,\nu}(z)$ it can be expressed as \cite[Eq. 11.10.17]{NIST:DLMF}
\begin{equation}  \label{44}
\mathbf{J}_{\nu}\left(z\right)=
\mathbf{j}_{0}(\nu)s_{0,\nu}(z) + \mathbf{j}_{-1}(\nu) s_{-1,\nu}(z).
\end{equation}

Next from (\ref{06}), (\ref{07}), (\ref{15}), (\ref{23}), (\ref{34}), (\ref{43}) and (\ref{44}) we have the more useful four identities
\begin{equation}  \label{48}
\mathbf{J}_{\nu}(z) = 
\mathbf{j}_{0}(\nu)S_{0,\nu}^{(j)}(z)+\mathbf{j}_{-1}(\nu)S_{-1,\nu}^{(j)}(z)
+J_{\nu}(z) \quad (j=0,1,2,\emptyset),
\end{equation}
where $j=\emptyset$ denotes $S_{\mu,\nu}(z)$.
As $z \rightarrow 0$ we have for $\nu \neq 0$
\begin{equation}  \label{45}
\mathbf{J}_{\nu}(z)=\frac{\sin(\pi \nu)}{\pi \nu}+\mathcal{O}(z),
\end{equation}
and if $\nu \in \mathbb{N}$ it reduces to being equal to $J_{\nu}(z)$, in which case of course (\ref{46}) applies instead.

Next the Weber function is defined by
\begin{equation}  \label{50}
\mathbf{E}_{\nu}(z)=\frac{1}{\pi}\int_{0}^{\pi}
\sin(\nu \theta-z\sin(\theta))d\theta,
\end{equation}
and this satisfies (\ref{42}) with $\mathbf{j}_{l}(\nu)$ replaced by $\mathbf{e}_{l}(\nu)$ ($l=0,-1$) where
\begin{equation}  \label{51}
\mathbf{e}_{0}(\nu)=-\frac{1}{\pi} \left\{1+\cos(\pi\nu)\right\}, 
\quad
\mathbf{e}_{-1}(\nu)=-\frac{\nu}{\pi}\left\{1-\cos(\pi\nu)\right\}.
\end{equation}
From \cite[Eq. 11.10.18]{NIST:DLMF}
\begin{equation}  \label{52}
\mathbf{E}_{\nu}(z)=\mathbf{e}_{0}(\nu)s_{0,\nu}(z)
+\mathbf{e}_{-1}(\nu) s_{-1,\nu}(z).
\end{equation}
For $j=1,2$ we have from (\ref{15}), (\ref{23}), (\ref{34}) and (\ref{52}) the more useful relations
\begin{equation}  \label{53}
\mathbf{E}_{\nu}(z) = \mathbf{e}_{0}(\nu)S_{0,\nu}^{(j)}(z)
+\mathbf{e}_{-1}(\nu)S_{-1,\nu}^{(j)}(z)+(-1)^{j}iJ_{\nu}(z),
\end{equation}
and hence from (\ref{34}) we also have
\begin{equation}  \label{54}
\mathbf{E}_{\nu}(z) = \mathbf{e}_{0}(\nu)S_{0,\nu}^{(0)}(z)
+\mathbf{e}_{-1}(\nu)S_{-1,\nu}^{(0)}(z).
\end{equation}
Also from (\ref{06}) and (\ref{52})
\begin{equation}  \label{55}
\mathbf{E}_{\nu}(z) = \mathbf{e}_{0}(\nu) S_{0,\nu}(z)
+\mathbf{e}_{-1}(\nu) S_{-1,{\nu}}(z) - Y_{\nu}(z).
\end{equation}
and as $z \rightarrow 0$ we have for $\nu \neq 0$
\begin{equation}  \label{55a}
\mathbf{E}_{\nu}(z)=\frac{1-\cos(\pi \nu)}{\pi \nu}+\mathcal{O}(z).
\end{equation}

The associated Anger-Weber function is defined by
\begin{equation}  \label{56}
\mathbf{A}_{\nu}(z)=\frac{1}{\pi}\int_{0}^{\infty}\exp(-\nu t-z\sinh(t))dt,
\end{equation}
and satisfies (\ref{42}) with $\mathbf{j}_{0}(\nu)$ and $\mathbf{j}_{-1}(\nu)$ replaced by $1/\pi$ and $-\nu/\pi$, respectively. Now from (\cite[Eq. 11.11.4]{NIST:DLMF})
\begin{equation}  \label{56a}
\mathbf{A}_{\nu}(z) \sim 
\frac{1}{\pi z}
\quad (z\rightarrow \infty, \,|\arg(z)|\leq\pi-\delta),
\end{equation}
and hence on referring to (\ref{07a}) and (\ref{08}) we assert that
\begin{equation}  \label{59}
\mathbf{A}_{\pm \nu}(z) = \frac{1}{\pi}\left \{S_{0,\nu}(z)
\mp \nu S_{-1,\nu}(z) \right\},
\end{equation}
since both sides of this equation are the unique solutions of the same ODE that are subdominant as $z \rightarrow \infty \exp(\pm \pi i/2)$. From (\ref{13}), (\ref{22}), (\ref{34}) and (\ref{59}) it then follows that for $j=0,1,2$
\begin{equation}  \label{60}
\mathbf{A}_{\nu}(z)
= \frac{1}{\pi}\left \{S_{0,\nu}^{(j)}(z)
-\nu S_{-1,\nu}^{(j)}(z) \right\}.
\end{equation}

Extensions of our asymptotic expansions to values of $z$ outside the given regions, including across the cut $(-\infty,0]$ follow from the following analytic continuation and connection formulas. 

Firstly, from (\ref{07}), (\ref{29}), (\ref{59}) and (\ref{60}) and for $m \in \mathbb{Z}$
\begin{equation}   \label{60a}
\mathbf{A}_{\nu}\left(ze^{m\pi i}\right)
=\mathbf{A}_{\nu}(z)
+\frac{1-e^{m \pi \nu i}}{\sin(\nu \pi)}J_{\nu}(z),
\end{equation}
when $m$ is even, and
\begin{multline}   \label{60b}
\mathbf{A}_{\nu}\left(ze^{m\pi i}\right)
= -\frac{1}{\pi}\left \{S_{0,\nu}^{(1)}(z)
+\nu S_{-1,\nu}^{(1)}(z) \right\}
+\frac{e^{-\pi \nu i}-e^{m \pi \nu i}}{\sin(\nu \pi)}J_{\nu}(z)  \\
= -\frac{1}{\pi}\left \{S_{0,\nu}^{(2)}(z)
+\nu S_{-1,\nu}^{(2)}(z) \right\}
+\frac{e^{\pi \nu i}-e^{m \pi \nu i}}{\sin(\nu \pi)}J_{\nu}(z),
\end{multline}
when $m$ is odd. When $\nu$ is an integer the limiting values are taken for the coefficients of $J_{\nu}(z)$, using L'H\^{o}pital's rule. Note in both cases the analytic continuation of $\mathbf{A}_{\nu}(z)$ remains bounded at $z=0$. When $m$ is even it is unbounded at infinity except when the coefficient of $J_{\nu}(z)$ in (\ref{60a}) is zero. When $m$ is odd is unbounded at infinity except in the upper half plane when the coefficient of the first $J_{\nu}(z)$ term in (\ref{60b}) is zero, and in the lower half plane when the coefficient of the second $J_{\nu}(z)$ term in this equation is zero.

Important inter-relations between the three functions read as follows, and we shall use these too. Firstly from \cite[Eq. 11.10.15]{NIST:DLMF}
\begin{equation}  \label{57}
\mathbf{J}_{\nu}(z)=
\sin(\pi\nu) \mathbf{A}_{\nu}(z)+J_{\nu}(z).
\end{equation}

Next from (\ref{51}), (\ref{54}) and (\ref{60})
\begin{equation}  \label{58}
\mathbf{E}_{\nu}(z) = 
-\cos(\pi\nu) \mathbf{A}_{\nu}(z)
-\frac{1}{\pi}\left \{S_{0,\nu}^{(0)}(z)
+\nu S_{-1,\nu}^{(0)}(z) \right\}.
\end{equation}

Finally using \cite[Eqs. 11.10.12 - 11.10.14]{NIST:DLMF}
\begin{equation}  \label{117}
\mathbf{J}_{-\nu}(z)=\mathbf{J}_{\nu}(-z)
=\sin(\pi\nu)\mathbf{E}_{\nu}(z)
+\cos(\pi\nu)\mathbf{J}_{\nu}(z),
\end{equation}
and
\begin{equation}  \label{118}
\mathbf{E}_{-\nu}(z)=-\mathbf{E}_{\nu}(-z)
=\cos(\pi\nu)\mathbf{E}_{\nu}(z)
-\sin(\pi\nu)\mathbf{J}_{\nu}(z).
\end{equation}

\subsection{Asymptotic expansions} We begin by defining for $s=0,1,2,\cdots$ two sequences of functions by
\begin{equation}  \label{99}
G_{s}^{\pm}(z)=z^{-1/2}\left\{G_{0,s}(z) \pm G_{-1,s}(z)\right\},
\end{equation}
where $G_{\mu,s}(z)$ are given by (\ref{83}) and (\ref{84}). Consequently
\begin{equation}  \label{100}
G_{0}^{+}(z)=\frac{G_{0,0}(z)+G_{-1,0}(z)}{z^{1/2}}
=\frac{1}{z-1},
\end{equation}
and 
\begin{equation}  \label{101}
G_{0}^{-}(z)=\frac{G_{0,0}(z)-G_{-1,0}(z)}{z^{1/2}}
=\frac{1}{z+1}.
\end{equation}
Using (\ref{84}) and (\ref{99}) we find that the rest satisfy the recursion relation
\begin{equation}  \label{102}
G_{s+1}^{\pm}(z)
=\frac{z {G_{s}^{\pm}}'(z)+z^{2}{G_{s}^{\pm}}''(z)}
{1-z^{2}}
\quad (s=0,1,2,\cdots).
\end{equation}
Thus by induction $G_{s}^{+}(z)=-G_{s}^{-}(-z)$. For the next three $G_{s}^{-}(z)$ coefficients we get from (\ref{100}) and (\ref{102})
\begin{equation}  \label{103}
G_{1}^{-}(z)
=-\frac{z}{(z+1)^{4}},
\end{equation}
\begin{equation}  \label{104}
G_{2}^{-}(z)
=\frac{z(9z-1)}{(z+1)^{7}},
\end{equation}
and
\begin{equation}  \label{105}
G_{3}^{-}(z)
=-\frac{z(225z^{2} - 54z + 1)}{(z+1)^{10}}.
\end{equation}
These suggest that $G_{s}^{-}(z)$ only have poles at $z=-1$, but this is not so obvious from the recursion formula (\ref{102}). However this is verified as part of the following.
\begin{theorem}
The coefficients $G_{s}^{-}(z)$ are rational functions of $z$ whose only poles are at $z=-1$, and as $\nu \rightarrow \infty$ uniformly for $z \in S(\delta)$
\begin{equation}  \label{106}
\mathbf{A}_{\nu}(\nu z)
\sim  \frac{1}{\pi}
\sum\limits_{s=0}^{\infty} \frac{G_{s}^{-}(z)}{\nu^{2s+1}}.
\end{equation}
\end{theorem}

\begin{proof}
Firstly it is clear from their recursion relation that  $G_{s}^{-}(z)$ are rational functions whose only possible poles are at $z=\pm 1$. Next from (\ref{89a}) - (\ref{89c}), (\ref{59}) and (\ref{60}) we have for $(j,k)=(-1,0)$, $(0,1)$ and $(-1,1)$
\begin{equation}  \label{107}
\mathbf{A}_{\nu}(\nu z)
\sim \frac{\nu}{\pi z^{1/2}} \left\{
w_{0}^{(j,k)}(\nu,z) - w_{-1}^{(j,k)}(\nu,z)
\right\} \quad (\nu \rightarrow \infty, \, z \in S^{(j,k)}(\delta)).
\end{equation}
Hence from (\ref{85})
\begin{equation}  \label{108}
\mathbf{A}_{\nu}(\nu z)
\sim \frac{1}{\pi z^{1/2} }
\sum\limits_{s=0}^{\infty} \frac{G_{0,s}(z) - G_{-1,s}(z)}{\nu^{2s+1}}.
\end{equation}
which from (\ref{99}) yields (\ref{106}) for $z$ lying in the union of $S^{(j,k)}(\delta)$ ($(j,k)=(-1,0)$, $(0,1)$ and $(-1,1)$). On examining \cref{fig:fig2,fig:fig3,fig:fig4} we see that this union contains $S(\delta')$ but excludes part of a closed disk of radius $\delta'$ centred at $z=1$, where $\delta'$ is slightly larger than $\delta$ but is $\mathcal{O}(\delta)$ as $\delta \rightarrow 0$. Extension of (\ref{106}) into the (temporarily) excluded disk follows from \cref{thm:nopoles}, which also establishes analyticity of the rational functions $G_{s}^{-}(z)$ at $z=1$. Thus we have established that (\ref{106}) is true for $z \in S(\delta')$. Finally since $\delta$ is an arbitrary small positive constant we can replace $S(\delta')$ by $S(\delta)$, and the result follows.
\end{proof}

We remark that Watson \cite[Sect. 10.15]{Watson:1944:TTB} obtained an expansion similar to (\ref{106}) with coefficients that are the same, via integral methods. In \cite{Meijer:1932:UDA} explicit formulas were given for ($s$)th coefficient in terms of the ($2s$)th derivative of a certain function. In \cite{Nemes:2014:RP1} and \cite{Nemes:2014:RP2} the coefficients were expressed as a double sum of generalised Bernoulli polynomials, a recursion formula equivalent to (\ref{102}) was derived, and bounds for the error terms were given. In all of these the expansions (in terms of our notation) $z$ is positive and $\nu$ is complex. Thus they are restricted to the phase of the order and argument being the same. In contrast we consider real $\nu$, and the present results have been demonstrated to be uniformly valid for complex $z$ lying in a domain containing the turning point $z=1$, as well as $z=0$ and $z=\infty$ ($|\arg(z)|\leq \pi -\delta$), in accord with the powerful results for Bessel functions that currently exist.

The same is true for the following new expansions for negative order $-\nu$, which is a more complicated situation. Using (\ref{92}), (\ref{93}), (\ref{96}), (\ref{59}) and (\ref{99}) we have:
\begin{theorem}
As $\nu \rightarrow \infty$ uniformly for $z \in S(\delta)$
\begin{equation}  \label{111}
\mathbf{A}_{-\nu}(\nu z) \sim
\frac{\sqrt{2}}{\nu^{1/3}z^{1/2}}\left\{\mathrm{Hi}
\left(\nu^{2/3}\zeta  \right)\mathcal{A}(\nu,z)+
{\mathrm{Hi}}^{\prime}\left(\nu^{2/3}\zeta \right)
\mathcal{B}(\nu,z) \right\}
+\mathcal{J}(\nu,z),
\end{equation}
where
\begin{equation}  \label{112}
\mathcal{J}(\nu,z)=
\frac{1}{\pi }
\sum\limits_{s=0}^{\infty}
\frac{G_{s}^{+}(z)}{\nu^{2s+1}}
-\frac{\sqrt{2}\phi(z)J(\nu,z)}{\pi \nu \zeta},
\end{equation}
$\zeta$ is given by (\ref{62}), and expansions for $\mathcal{A}(\nu,z)$, $\mathcal{B}(\nu,z)$ and $J(\nu,z)$ are furnished by (\ref{76}), (\ref{77}) and (\ref{94}).
\end{theorem}

We remark that, unlike $G_{s}^{-}(z)$, $G_{s}^{+}(z)$ have poles at $z=1$, and these are cancelled by corresponding singularities of $J(\nu,z)$ when $\mathcal{J}(\nu,z)$ is expanded as an asymptotic series in inverse powers of $\nu$ (again by applying \cref{thm:nopoles}). Thus for $z$ close to or equal to $1$, either a few such terms in a re-expansion can be used, or for more than a few terms in (\ref{112}), we use
\begin{equation}  \label{113}
\mathcal{J}(\nu,z) \sim
\frac{1}{2\pi i}\oint_{C}
\left\{\sum\limits_{s=0}^{\infty}
\frac{G_{s}^{+}(t)}{\nu^{2s+1}}
-\frac{\sqrt{2}\phi(t)J(\nu,t)}{\pi \nu \zeta(t)} \right\}
\frac{dt}{t-z},
\end{equation}
where $C$ is a positively orientated circle centred at $t=1$ of positive radius $r$ such that $|z-1|<r<1$; this also applies to (\ref{115}), (\ref{119}) and (\ref{120}) below. As mentioned in \cref{sec3} this is also how $\mathcal{A}(\nu,z)$ and $\mathcal{B}(\nu,z)$ are computed in the same circumstances.

We also note that unlike $\mathbf{A}_{\nu}(\nu z)$, $\mathbf{A}_{-\nu}(\nu z)$ is unbounded at $z=0$, and this is reflected in (\ref{111}) by the fact that $\mathrm{Hi}(\nu^{2/3}\zeta)$ is unbounded as $z \rightarrow 0+$ ($\zeta \rightarrow +\infty$). In fact we see from (\ref{111}) that $\mathbf{A}_{-\nu}(\nu z)$ is unbounded in the interior of $S_{0}$ as $\nu \rightarrow \infty$.

Next for the Anger and Weber functions we obtain the following uniform asymptotic expansions.
\begin{theorem}
As $\nu \rightarrow \infty$ uniformly for $z \in S(\delta)$
\begin{equation}  \label{114}
\mathbf{J}_{\nu}(\nu z) \sim
\frac{\sin(\pi\nu)}{\pi}
\sum\limits_{s=0}^{\infty} \frac{G_{s}^{-}(z)}{\nu^{2s+1}}
+\frac{\sqrt{2} \, w_{0}(\nu,z)}{\nu^{1/3}},
\end{equation}
and
\begin{multline}  \label{115}
\mathbf{E}_{\nu}(\nu z) \sim
\frac{1}{\nu^{1/3}}\left\{\mathrm{Gi}\left(\nu^{2/3}\zeta  \right)\mathcal{A}(\nu,z)+
{\mathrm{Gi}}^{\prime}\left(\nu^{2/3}\zeta \right)
\mathcal{B}(\nu,z) \right\} 
\\
-\frac{\cos(\pi\nu)}{\pi }
\sum\limits_{s=0}^{\infty} 
\frac{G_{s}^{-}(z)}{\nu^{2s+1}}
- \mathcal{J}(\nu,z),
\end{multline}
where $w_{0}(\nu,z)$ is given by (\ref{79}).
\end{theorem}

\begin{remark} \label{remark2}
In place of (\ref{115}) when $|z|$ is small one can use the more stable expansion (cf. \cref{remark1})
\begin{equation}  \label{116}
\mathbf{E}_{\nu}(\nu z) \sim
-\frac{1}{\pi } \sum\limits_{s=0}^{\infty} 
\frac{G_{s}^{+}(z)
+\cos(\pi\nu)G_{s}^{-}(z)}{\nu^{2s+1}}.
\end{equation}
This comes from (\ref{34}), (\ref{85}), (\ref{88}), (\ref{89}), (\ref{51}) and (\ref{54}).
\end{remark}

\begin{proof}
The expansion (\ref{114}) follows from (\ref{80}), (\ref{57}) and (\ref{106}). Using (\ref{92}), (\ref{98}), (\ref{58}), (\ref{106}), and (\ref{112}) yields (\ref{115}).
\end{proof}

For negative order $-\nu$ we have the following expansions. Unlike the Anger-Weber function $\mathbf{A}_{-\nu}(\nu z)$ the two functions here remain bounded at $z=0$, and hence like (\ref{116}) we include alternative simpler expansions which are stable for small $|z|$. Also note from the first of (\ref{117}) and (\ref{118}) that the proceeding results provide expansions for positive order $\nu$ in the whole left half plane.

\begin{theorem}
As $\nu \rightarrow \infty$ uniformly for $z \in S(\delta)$
\begin{multline}  \label{119}
\mathbf{J}_{-\nu}(\nu z) \sim
\frac{\sin(\pi\nu)}{\nu^{1/3}}\left\{\mathrm{Gi}\left(\nu^{2/3}\zeta  \right)\mathcal{A}(\nu,z)+
{\mathrm{Gi}}^{\prime}\left(\nu^{2/3}\zeta \right)
\mathcal{B}(\nu,z) \right\} 
\\
- \sin(\pi\nu)\mathcal{J}(\nu,z)
+\frac{\sqrt{2} \cos(\pi\nu) w_{0}(\nu,z)}{\nu^{1/3}},
\end{multline}
and
\begin{multline}  \label{120}
\mathbf{E}_{-\nu}(\nu z) \sim
\frac{\cos(\pi\nu)}{\nu^{1/3}}\left\{\mathrm{Gi}\left(\nu^{2/3}\zeta  \right)\mathcal{A}(\nu,z)+
{\mathrm{Gi}}^{\prime}\left(\nu^{2/3}\zeta \right)
\mathcal{B}(\nu,z) \right\}    \\
- \frac{1}{\pi }\sum\limits_{s=0}^{\infty}
\frac{G_{s}^{-}(z)}{\nu^{2s+1}}
- \cos(\pi\nu)\mathcal{J}(\nu,z)
-\frac{\sqrt{2} \sin(\pi\nu) w_{0}(\nu,z)}{\nu^{1/3}}.
\end{multline}
Near $z=0$ the following alternative stable expansions hold
\begin{equation}  \label{121}
\mathbf{J}_{-\nu}(\nu z) \sim
-\frac{\sin(\pi\nu)}{\pi }\sum\limits_{s=0}^{\infty} 
\frac{G_{s}^{+}(z)}{\nu^{2s+1}}
+\frac{\sqrt{2} \cos(\pi\nu) w_{0}(\nu,z)}{\nu^{1/3}},
\end{equation}
and
\begin{equation}  \label{122}
\mathbf{E}_{-\nu}(\nu z) \sim
-\frac{1}{\pi } \sum\limits_{s=0}^{\infty} 
\frac{\cos(\pi\nu)G_{s}^{+}(z)
+G_{s}^{-}(z)}{\nu^{2s+1}}
-\frac{\sqrt{2} \sin(\pi\nu) w_{0}(\nu,z)}{\nu^{1/3}}.
\end{equation}
\end{theorem}

\begin{proof}
Expansions (\ref{119}) and (\ref{120}) follow from (\ref{117}), (\ref{118}), (\ref{114}) and (\ref{115}). Expansions (\ref{121}) and (\ref{122}) follow similarly, except using (\ref{116}) instead of (\ref{115}).
\end{proof}

We finish this section by considering the so-called Neumann polynomials, which are important because they are used to expand functions in terms of Bessel functions. These are defined by $O_{0}(z)=1/z$ and \cite[Eq. 10.23.13]{NIST:DLMF}
\begin{equation}  \label{96aa}
O_{n}(z)=\frac{n}{4}
\sum_{k=0}^{\left\lfloor n/2\right\rfloor}
\frac{(n-k-1)!}{k!}\left(\frac{2}{z}\right)^{n-2k+1}
\quad (n=1,2,3,\cdots).
\end{equation}
Evidently these are not polynomials in $z$, but rather in terms of $1/z$. For large order we have the uniform asymptotic expansions:
\begin{theorem}
\begin{multline}  \label{96a}
O_{n}(nz) \sim \frac{1}{z^{3/2} }
\left[\sum\limits_{s=0}^{\infty} \frac{G_{\mu,s}(z)}{n^{2s+1}} 
-\left(\frac{z}{2}\right)^{1/2}
\frac{\phi(z)J(n,z)}{n\zeta}
\right. \\
 +\frac{\pi}{\sqrt{2}\,n^{1/3}}
\left\{\mathrm{Hi}\left(n^{2/3}\zeta  \right)
\mathcal{A}(n,z)+
{\mathrm{Hi}}^{\prime}\left(n^{2/3}\zeta \right)
\mathcal{B}(n,z) \right\}
\Bigg],
\end{multline}
as $n \rightarrow \infty$ uniformly for $z \in S(\delta)$, where $G_{\mu,s}(z)$ are given by (\ref{83}) and (\ref{84}), and $\mu=1$ for $n$ even and $\mu=0$ for $n$ odd.
\end{theorem}

\begin{proof}
From \cite[Sect. 10.74]{Watson:1944:TTB} for $m=0,1,2,\cdots$
\begin{equation}  \label{96b}
O_{2m}(z)=z^{-1}S_{1,2m}(z), \quad 
O_{2m+1}(z)=(2m+1)z^{-1}S_{0,2m+1}(z).
\end{equation}
Then use (\ref{91}), (\ref{93}) and (\ref{96}).
\end{proof}

Since $O_{n}(z)$ is even/odd according to $n$ being odd/even respectively it suffices to use (\ref{96a}) for $\Re(z) \geq 0$ ($z \neq 0$).

\section{Struve functions}
\label{sec5}

Struve functions are a special case of Lommel functions. From \cite[Sect. 11.2]{NIST:DLMF} they satisfy the inhomogeneous differential equation
\begin{equation}  \label{123}
\frac{d^{2}y}{dz^{2}}+\frac{1}{z}\frac{dy}{dz}
+\left ( 1-\frac{\nu^{2}}{z^{2}} \right )y=
\frac{\left(\tfrac{1}{2}z\right)^{\nu-1}}
{\sqrt{\pi}\Gamma\left (\nu+\tfrac{1}{2} \right )}.
\end{equation}
Thus if a Lommel function $y_{\mu,\nu}(z)$ is a solution of (\ref{02}) then 
\begin{equation}  \label{123a}
\frac{y_{\nu,\nu}(z)}
{2^{\nu-1}\sqrt{\pi}\Gamma\left(\nu+\tfrac{1}{2} \right )},
\end{equation}
is a solution of (\ref{123}). However, we cannot directly apply the asymptotic expansions of \cref{sec3} since in the present case $\mu=\nu$ and hence is unbounded. Before we address this, let us recap properties of fundamental solutions from the literature, along with new definitions for numerically satisfactory solutions.

The first important solution is defined by \cite[11.2.1]{NIST:DLMF}
\begin{equation}  \label{124}
\mathbf{H}_{\nu}(z)=
\left(\tfrac{1}{2}z\right)^{\nu+1}\sum_{k=0}^{\infty}
\frac{(-1)^{k}\left(\tfrac{1}{2}z\right)^{2k}}
{\Gamma\left(k+\tfrac{3}{2}\right)
\Gamma\left(k+\nu+\tfrac{3}{2}\right)},
\end{equation}
where principal values correspond to principal values of $(z/2)^{\nu+1}$. This has the characteristic property of being bounded at $z=0$ and real for $z$ positive. Unlike the corresponding Lommel function $s_{\mu,\nu}(z)$ this is well-defined for all positive $\nu$, so it is not necessary for us to introduce another real-valued solution subdominant at $z=0$, as we did with (\ref{34}).

A numerically satisfactory companion is defined by \cite[11.2.5]{NIST:DLMF}
\begin{equation}  \label{125}
\mathbf{K}_{\nu}(z)=\mathbf{H}_{\nu}(z)-Y_{\nu}(z).
\end{equation}
This is important due to the uniquely-defining property from \cite[Eq. 11.6.1]{NIST:DLMF} that as $z \rightarrow \infty$ with $|\arg(z)| \leq \pi - \delta$ 
\begin{equation}  \label{126}
\mathbf{K}_{\nu}(z)\sim\frac{1}{\pi}
\sum_{k=0}^{\infty}\frac{\Gamma\left(k+\tfrac{1}{2}\right)
\left(\tfrac{1}{2}z\right)^{\nu-2k-1}}
{\Gamma\left(\nu+\tfrac{1}{2}-k\right)}.
\end{equation}

From variation of parameters on (\ref{123}) (cf. (\ref{10}))
\begin{multline}  \label{127}
\mathbf{K}_{\nu}(z)
=\frac{i\sqrt{\pi}}
{2^{\nu+1}\Gamma\left (\nu+\tfrac{1}{2} \right )}
\left[H_{\nu}^{(2)}(z)\int_{\infty \exp(\pi i/2)}^z 
t^{\nu}H_{\nu}^{(1)}(t)dt \right.
\\
\left. -H_{\nu}^{(1)}(z)\int_{\infty \exp(-\pi i/2)}^z 
t^{\nu}H_{\nu}^{(2)}(t)dt
\right].
\end{multline}

Analogously to (\ref{13}) and (\ref{22}) we define
\begin{equation}  \label{128}
\mathbf{K}_{\nu}^{(1)}(z)
=\mathbf{K}_{\nu}(z)-iH_{\nu}^{(1)}(z)
=\mathbf{H}_{\nu}(z)-iJ_{\nu}(z),
\end{equation}
and
\begin{equation}  \label{129}
\mathbf{K}_{\nu}^{(2)}(z)
=\mathbf{K}_{\nu}(z)+iH_{\nu}^{(2)}(z)
=\mathbf{H}_{\nu}(z)+iJ_{\nu}(z).
\end{equation}
From these definitions, along with (\ref{46}), (\ref{46a}), (\ref{46b}), (\ref{124}) and (\ref{126}), we see that $\mathbf{K}_{\nu}^{(j)}(z)$ is uniquely defined by being subdominant at $(-1)^{j-1} i \infty$ in the principal plane, as well as at $z=0$. We note that from (\ref{128}) and (\ref{129})
\begin{equation}  \label{136a}
\mathbf{H}_{\nu}(z)
=\tfrac{1}{2}\left\{\mathbf{K}_{\nu}^{(1)}(z)
+\mathbf{K}_{\nu}^{(2)}(z)\right\}.
\end{equation}

In terms of Lommel functions $\mathbf{K}_{\nu}(z)$, $\mathbf{H}_{\nu}(z)$, $\mathbf{K}_{\nu}^{(1)}(z)$ and $\mathbf{K}_{\nu}^{(2)}(z)$ are given by (\ref{123a}) with $y_{\nu,\nu}(z)$ equal to $S_{\nu,\nu}(z)$, $S_{\nu,\nu}^{(0)}(z)$, $S_{\nu,\nu}^{(1)}(z)$ and $S_{\nu,\nu}^{(2)}(z)$, respectively. This is because the corresponding functions are solutions of (\ref{123}) having the same unique subdominant behaviour at two of the singularities, except for the $\mathbf{H}_{\nu}(z)$ identification which comes from (\ref{34}) and (\ref{136a}).

From \cite[Eqs. 10.11.1 and 11.4.16]{NIST:DLMF} and the above definitions we also have the following analytic continuation formulas for $m \in \mathbb{Z}$
\begin{equation}  \label{130}
\mathbf{K}_{\nu}^{(j)}(ze^{2m\pi i})
=e^{2m \nu \pi i}\mathbf{K}_{\nu}^{(j)}(z)
\quad (j=1,2),
\end{equation}
\begin{equation}  \label{130a}
\mathbf{K}_{\nu}^{(1)}(ze^{(2m+1)\pi i})
=-e^{(2m+1) \nu \pi i}\mathbf{K}_{\nu}^{(2)}(z),
\end{equation}
and
\begin{equation}  \label{130b}
\mathbf{K}_{\nu}^{(2)}(ze^{(2m+1)\pi i})
=-e^{(2m+1) \nu \pi i}\mathbf{K}_{\nu}^{(1)}(z).
\end{equation}

We now turn to the problem of obtaining asymptotic expansions. To do so we shall express the Struve functions in terms of Lommel functions of bounded $\mu$, which we call $\tilde{\mu}$. To define this we first introduce the large positive integer
\begin{equation}  \label{131}
k_{\nu}=\lfloor \tfrac{1}{2}\nu -\tfrac{1}{2} \rfloor.
\end{equation}
Observe that $0 \leq \nu-2k_{\nu}-1 < 2$. Then our bounded parameter satisfies $-1 \leq \tilde{\mu} < 1$ and is defined by
\begin{equation}  \label{132}
\tilde{\mu}=\nu-2k_{\nu}-2.
\end{equation}

We next define a function $p_{\nu}(z)$ to consist of the first $k_{\nu}+1$ terms of (\ref{126}), and in particular terms involving nonnegative powers of $z$, by
\begin{equation}  \label{133}
p_{\nu}(z)=\frac{1}{\pi}
\sum_{k=0}^{k_{\nu}}\frac{\Gamma\left(k+\tfrac{1}{2}\right)
\left(\tfrac{1}{2}z\right)^{\nu-2k-1}}
{\Gamma\left(\nu+\tfrac{1}{2}-k\right)}.
\end{equation}
Our desired representations are then given by:
\begin{lemma} \label{lomstr}
\begin{equation}  \label{135}
\mathbf{K}_{\nu}(z)
=B(\tilde{\mu},\nu)S_{\tilde{\mu},\nu}(z)+p_{\nu}(z),
\end{equation}
\begin{equation}  \label{136}
\mathbf{K}_{\nu}^{(j)}(z)
=B(\tilde{\mu},\nu) S_{\tilde{\mu},\nu}^{(j)}(z)+p_{\nu}(z)
\quad (j=1,2),
\end{equation}
and
\begin{equation}  \label{136b}
\mathbf{H}_{\nu}(z)
=B(\tilde{\mu},\nu) S_{\tilde{\mu},\nu}^{(0)}(z)+p_{\nu}(z),
\end{equation}
where
\begin{equation}  \label{134b}
B(\tilde{\mu},\nu)
=\frac{2^{1-\tilde{\mu}}
\Gamma\left(\tfrac{1}{2}\nu-\tfrac{1}{2}\tilde{\mu}+\tfrac{1}{2} \right)}{\pi \Gamma\left(\tfrac{1}{2}\nu
+\tfrac{1}{2}\tilde{\mu}+\tfrac{1}{2}
\right )}.
\end{equation}
\end{lemma}

\begin{proof}
By direct substitution of (\ref{133}) into (\ref{123}) it is straightforward to verify that $p_{\nu}(z)$ satisfies the modified Lommel equation
\begin{equation}  \label{134}
\frac{d^{2}p_{\nu}}{dz^{2}}+\frac{1}{z}\frac{dp_{\nu}}{dz}
+\left ( 1-\frac{\nu^{2}}{z^{2}} \right )p_{\nu}
=\frac{\left(\tfrac{1}{2}z\right)^{\nu-1}}
{\sqrt{\pi}\Gamma\left (\nu+\tfrac{1}{2} \right )}
-\frac{\Gamma\left (k_{\nu}+\tfrac{3}{2} \right)
\left(\tfrac{1}{2}z\right)^{\tilde{\mu}-1}}
{\pi \Gamma\left (\nu-k_{\nu}-\tfrac{1}{2} \right )}.
\end{equation}
For the second term on the RHS we have on using (\ref{132})
\begin{equation}  \label{134a}
\frac{\Gamma\left (k_{\nu}+\tfrac{3}{2} \right)
\left(\tfrac{1}{2}z\right)^{\tilde{\mu}-1}}
{\pi \Gamma\left (\nu-k_{\nu}-\tfrac{1}{2} \right )}
=B(\tilde{\mu},\nu)z^{\tilde{\mu}-1},
\end{equation}
where $B(\tilde{\mu},\nu)$ is given by (\ref{134b}). Then from (\ref{02}) (with $\mu = \tilde{\mu}$), (\ref{08}), (\ref{126}), (\ref{134}) and (\ref{134a}) we deduce (\ref{135}) must hold, since both sides satisfy (\ref{123}) and have the same unique subdominant behaviour at $z = \pm i\infty$ when $|\arg(z)|\leq \pi -\delta$. The relations (\ref{136}) can be established by similar reasoning. Finally from (\ref{34}), (\ref{136a}) and (\ref{136}) we arrive at (\ref{136b}).
\end{proof}

Uniform asymptotic expansions for the four Struve functions now follow from \cref{lomstr} and the ones for the corresponding Lommel functions given in \cref{sec3}. So in (\ref{135}) we simply replace $z$ by $\nu z$ for all three functions, recall (\ref{133}) and (\ref{134b}), and then use the uniform asymptotic expansions (\ref{87}) and (\ref{96}) with $\mu$ replaced by $\tilde{\mu}$ (given by (\ref{131}) and (\ref{132})). Likewise for (\ref{136}) and (\ref{136b}) with the expansions for $S_{\tilde{\mu},\nu}^{(j)}(\nu z)$ ($j=0,1,2$) given by (\ref{88}), (\ref{89}), (\ref{89aa}), (\ref{89b}), (\ref{89c}), (\ref{95}) and (\ref{98}).

Although we have obtained asymptotic expansions that are valid for $z \in S(\delta)$ there is a problem near $z=0$ for the three that are subdominant at this point. To illustrate this and show how to overcome it, let us focus on  the expansion for $\mathbf{H}_{\nu}(\nu z)$, since the ones for $\mathbf{K}_{\nu}^{(j)}(\nu z)$ ($j=1,2$) are identical to this near $z=0$.

As we described in \cref{remark1}, we must use the expansion (\ref{89aa}) for $S_{\tilde{\mu},\nu}^{(0)}(z)$ in (\ref{136b}) for  $|z| < \delta$ (which we assume from now on). Here $\delta$ is an arbitrary positive constant, but it must be sufficiently small so that $z \in S_{0}$. Thus we have
\begin{equation}  \label{137}
\mathbf{H}_{\nu}(\nu z) \sim
\frac{\nu^{\tilde{\mu}} B(\tilde{\mu},\nu)}{z^{1/2}}
\sum\limits_{s=0}^{\infty} 
\frac{G_{\tilde{\mu},s}(z)}{\nu^{2s+1}}
+p_{\nu}(\nu z).
\end{equation}

The problem is that as $z \rightarrow 0$ both terms on the RHS of this relation behave like a constant times $z^{\tilde{\mu}+1}$ (see (\ref{34a}), (\ref{83}), (\ref{84}), (\ref{132}) and (\ref{133})). On the other hand from (\ref{124}) the LHS behaves like a constant times $z^{\nu+1}$, which is much smaller. We conclude there is cancellation of relatively large terms in the expansions for (\ref{136}) and (\ref{136b}), which renders the expansion numerically unsatisfactory near $z=0$ as it currently stands. To overcome this we proceed as follows.

Consider the first term on the RHS of (\ref{137}), and divide this by $z^{\nu+1}$. Then  using Stirling's formula \cite[Eq. 5.11.3]{NIST:DLMF} on (\ref{134b}) for large $\nu$ and bounded $\tilde{\mu}$ we can re-express this term formally as
\begin{equation}  \label{138}
\frac{\nu^{\tilde{\mu}} B(\tilde{\mu},\nu)}{z^{\nu+\frac{3}{2}}}
\sum\limits_{s=0}^{\infty} 
\frac{G_{\tilde{\mu},s}(z)}{\nu^{2s+1}}
=\frac{2}{\pi} \sum\limits_{s=0}^{\infty}
\frac{\tilde{G}_{\tilde{\mu},s}(z)}{\nu^{2s+1}},
\end{equation}
where from (\ref{132}) each $\tilde{G}_{\tilde{\mu},s}(z)$ has a pole of order at most $2k_{\nu}+2$ at $z=0$. For the first two we find
\begin{equation}  \label{139}
\tilde{G}_{\tilde{\mu},0}(z)=
\frac{G_{\tilde{\mu},0}(z)}{z^{\tilde{\mu}+2k_{\nu}+\frac{7}{2}}}
=\frac{1}{z^{2+2k_{\nu}}\left(z^2-1\right)},
\end{equation}
and
\begin{equation}  \label{140}
\tilde{G}_{\tilde{\mu},1}(z)=
\frac{G_{\tilde{\mu},1}(z) 
+ \frac{1}{6}\tilde{\mu}
\left(\tilde{\mu}^{2} - 1\right) G_{\tilde{\mu},0}(z)}
{z^{\tilde{\mu}+2k_{\nu}+\frac{7}{2}}}
=\frac{\tilde{g}_{\tilde{\mu},1}(z)}{6z^{2+2k_{\nu}}
\left(z^2-1\right)^{4}},
\end{equation}
where
\begin{multline}  \label{141}
\tilde{g}_{\tilde{\mu},1}(z)
=\tilde{\mu}\left(\tilde{\mu}^2-1 \right)z^{6}
-3(\tilde{\mu}-1) \left(\tilde{\mu}^{2}+3\tilde{\mu}-2\right)z^{4} \\
+3(\tilde{\mu}+3) \left(\tilde{\mu}^{2}+\tilde{\mu}-4\right)z^{2}
- (\tilde{\mu}+1)(\tilde{\mu}+2)(\tilde{\mu}+3).
\end{multline}

Now consider the second term on the RHS of (\ref{137}), also divided by $z^{\nu+1}$. On letting $z \rightarrow  \nu z$ and $k \rightarrow k_{\nu}-k$ in (\ref{133}) this can be expressed as
\begin{equation}  \label{142}
\frac{p_{\nu}(\nu z)}{z^{\nu+1}}
= \frac{1}{\pi z^{2+2k_{\nu}}}
\left(\frac{\nu}{2} \right)^{\tilde{\mu}+1}
\sum_{k=0}^{k_{\nu}}\frac{
\Gamma\left(\tfrac{1}{2}\nu-\tfrac{1}{2}\tilde{\mu}
-k-\tfrac{1}{2}\right)
\left(\tfrac{1}{2}\nu z\right)^{2k}}
{\Gamma\left(\tfrac{1}{2}\nu+\tfrac{1}{2}\tilde{\mu}
+k+\tfrac{3}{2}\right)}.
\end{equation}
If we expand this asymptotically in inverse powers of $\nu$, the poles at $z=0$ must \textit{exactly} cancel those of the RHS of (\ref{138}) when these two are added together, otherwise the RHS of (\ref{137}) would not be $\mathcal{O}(z^{\nu+1})$ as $z \rightarrow 0$. Moreover the RHS of (\ref{142}) is a finite sum of only negative integer powers of $z$, i.e. poles only. Consequently, as $\nu \rightarrow \infty$ with $|z|<\delta$ we arrive at our desired modification of (\ref{137}), namely from (\ref{138})
\begin{multline}  \label{143}
\mathbf{H}_{\nu}(\nu z)
\sim z^{\nu +1} \left\{\frac{\nu^{\tilde{\mu}} B(\tilde{\mu},\nu)}{z^{\nu+\frac{3}{2}}}
\sum\limits_{s=0}^{\infty} 
\frac{G_{\tilde{\mu},s}(z)}{\nu^{2s+1}}
+\frac{p_{\nu}(\nu z)}{z^{\nu +1}} \right\} \\
= z^{\nu +1} \left\{\frac{2}{\pi} \sum\limits_{s=0}^{\infty}
\frac{\tilde{G}_{\tilde{\mu},s}(z)}{\nu^{2s+1}}
+\frac{p_{\nu}(\nu z)}{z^{\nu +1}} \right\}
= \frac{2 z^{\nu +1}}{\pi} \sum\limits_{s=0}^{\infty}
\frac{\tilde{G}_{\tilde{\mu},s}^{\ast}(z)}{\nu^{2s+1}},
\end{multline}
where $\tilde{G}_{\tilde{\mu},s}^{\ast}(z)$ denotes the regular part of $\tilde{G}_{\tilde{\mu},s}(z)$ at $z=0$, i.e. the function with its poles at $z=0$ removed. Recall the same expansion as (\ref{143}) holds for $\mathbf{K}_{\nu}^{(j)}(z)$ ($j=1,2$) as $\nu \rightarrow \infty$ with $|z|<\delta$.

How do we compute the regular parts for all the coefficients in this expansion? For the first term (\ref{139}) we can simply use the geometric series to get
\begin{equation}  \label{144}
\tilde{G}_{\tilde{\mu},0}^{\ast}(z)
=\frac{1}{z^{2+2k_{\nu}}\left(z^2-1\right)}
+\sum_{k=1}^{k_{\nu}+1}\frac{1}{z^{2k}}
=\frac{1}{z^2-1}.
\end{equation}
To do so for general terms we note from (\ref{83}), (\ref{84}), (\ref{132}) and (\ref{138}) that each coefficient $\tilde{G}_{\tilde{\mu},s}(z)$ is a linear combination of rational functions of the form
\begin{equation}  \label{145}
q_{l,m}(\omega)=\frac{1}{\omega^l(1-\omega)^m},
\end{equation}
where $l \leq k_{\nu}+1$, and $\omega = z^2$. For example from (\ref{140}), (\ref{141}) and (\ref{145})
\begin{multline}  \label{146}
\tilde{G}_{\tilde{\mu},1}(z)
=\tfrac{1}{6}
\tilde{\mu}\left(\tilde{\mu}^2-1\right)q_{k_{\nu}-5,4}(\omega)
-\tfrac{1}{2}(\tilde{\mu}-1) \left(\tilde{\mu}^{2}+3\tilde{\mu}-2\right)q_{k_{\nu}-3,4}(\omega) \\
+\tfrac{1}{2}(\tilde{\mu}+3) \left(\tilde{\mu}^{2}+\tilde{\mu}-4\right)q_{k_{\nu}-1,4}(\omega)
-\tfrac{1}{6}(\tilde{\mu}+1)
(\tilde{\mu}+2)(\tilde{\mu}+3) q_{k_{\nu}+1,4}(\omega).
\end{multline}
Thus $\tilde{G}_{\tilde{\mu},1}^{\ast}(z)$ is the same expression but with each $q_{l,m}(\omega)$ replaced by its regular part $q_{l,m}^{\ast}(\omega)$. The same of course is true for the general coefficient $\tilde{G}_{\tilde{\mu},s}^{\ast}(z)$. So it remains to find the regular parts of the rational functions $q_{l,m}(\omega)$, and this can be achieved using the binomial expansion. Hence, assuming $|\omega|<1$, we have
\begin{equation}  \label{147}
\frac{1}{(1-\omega)^m}
=\sum_{k=0}^{\infty}\binom{m+k-1}{k}\omega^{k},
\end{equation}
and therefore for $l \geq 1$ we have from (\ref{145})
\begin{equation}  \label{148}
q_{l,m}^{\ast}(\omega)
=\sum_{k=l}^{\infty}\binom{m+k-1}{k}\omega^{k-l}
=\binom{m+l-1}{l} F(l+m,1;l+1;\omega),
\end{equation}
where $F$ is the hypergeometric function \cite[Eq. 15.2.1]{NIST:DLMF}. Clearly for $l \leq 0$ there is nothing to do, because $q_{l,m}^{\ast}(\omega)=q_{l,m}(\omega)$ since from (\ref{145}) we see it does not have a pole at $\omega=z=0$.

\section*{Acknowledgments}
Financial support from Ministerio de Ciencia e Innovaci\'on, Spain, project PGC2018-098279-B-I00 (MCIU/AEI/FEDER, UE) is acknowledged. 

\bibliographystyle{siamplain}
\bibliography{biblio}

\end{document}